\documentclass[12pt]{article}

\usepackage{amssymb,amsmath,amsthm,thmtools,mathtools,tikz}
\usepackage[hidelinks]{hyperref}
\usepackage[nameinlink]{cleveref}
\hypersetup{
 pdftitle={Graphs with Minimum Algebraic Connectivity II: Regular Graphs of Even Degree},
 pdfsubject={Algebraic connectivity of regular graphs of even degree},
 pdfauthor={Maryam Abdi and Ebrahim Ghorbani},
 pdfkeywords={algebraic connectivity, spectral gap, regular graphs, diameter}
}
\newtheorem{theorem}{Theorem}[section]
\newtheorem{corollary}[theorem]{Corollary}

\newtheorem{lemma}[theorem]{Lemma}

\newcommand{\x}{{\bf x}}
\newcommand{\y}{{\bf y}}
\newcommand{\z}{{\bf z}}
\newcommand{\bu}{{\bf u}}
\newcommand{\bv}{{\bf v}}
\newcommand{\1}{{\bf 1}}
\newcommand{\0}{{\bf 0}}

\newcommand{\dm}{{\rm diam}}
\newcommand{\G}{{\cal G}}

\newcommand{\e}{\epsilon}
\newcommand{\E}{\mathcal E}
\tikzstyle{vertex}=[circle, draw, inner sep=0pt, minimum size=3pt]
\newcommand{\vertex}{\node[vertex]}

\begin{document}
\title{Graphs with Minimum Algebraic Connectivity II: \\ Regular Graphs of Even Degree}
\author{ Maryam Abdi$^{\,\rm a}$ \quad  Ebrahim Ghorbani$^{\,\rm b}$
	\\[.3cm]
	{\sl\normalsize $^{\rm a}$School of Mathematics, Institute for Research in Fundamental Sciences (IPM),}\\
	{\sl\normalsize P. O. Box 19395-5746, Tehran, Iran }\\
	{\sl\normalsize $^{\rm b}$Hamburg University of Technology, Institute for Algorithms and Complexity, Germany} \\
	{\tt\small m.abdi@ipm.ir\qquad  ebrahim.ghorbani@tuhh.de} }
\date{}
\maketitle
\begin{abstract}
	Aldous and Fill (2002) conjectured the asymptotic  maximum relaxation time of a random walk on a connected regular graph. Since the relaxation time of a $d$-regular graph $G$ is $d/\mu(G)$, where $\mu(G)$ denotes its algebraic connectivity, this conjecture is closely related to the problem of minimizing algebraic connectivity among regular graphs.		
Guiduli and Mohar (1996) conjectured that, for every fixed minimum degree $\delta=d\ge3$ and all sufficiently large orders, graphs with minimum algebraic connectivity are path-like and, apart from bounded portions near their two ends, have a prescribed block structure. Abdi and Ghorbani (2024) proposed an analogous structural conjecture for $d$-regular graphs with minimum algebraic connectivity and fixed degree $d\ge3$.
In Part~I, we proved the Aldous--Fill conjecture, the Guiduli--Mohar conjecture, and for odd degrees, the Abdi--Ghorbani conjecture. In this paper, we settle the remaining even-degree case, thereby completing the structural characterization of regular graphs with minimum algebraic connectivity.
 We also prove that, for every fixed even $d\ge4$, the minimum algebraic connectivity at order $n$ is $2(d-2)\pi^2/n^2+O_d(n^{-3})$, and that every minimizing graph has diameter $3n/(d+1)+O_d(1)$. For every fixed even $d\ge6$, $d$-regular graphs whose algebraic connectivity is asymptotically minimum have asymptotically maximum diameter. Finally, we obtain a sharp normalized-gap bound for all even regular degrees, including degrees that grow with $n$.

\par\vspace{4mm}
\noindent {\bf Keywords:} Algebraic connectivity, Regular graph,
Relaxation time of random walk, Maximum diameter\\[.1cm]
\noindent {\bf AMS Mathematics Subject Classification\,(2020):}
05C50, 05C35
\end{abstract}

\section{Introduction}\label{sec:intro}

All graphs considered in this paper are simple, that is, undirected
without loops or multiple edges, and connected. The \emph{relaxation
time} of the random walk on a graph $G$ is $1/(1-\eta_2)$, where
$\eta_2$ is the second largest eigenvalue of its transition matrix
$\Delta^{-1}A$, where $A$ is the adjacency matrix and $\Delta$ is
the diagonal matrix of vertex degrees. If $G$ is $d$-regular, then
$ \tau(G)=d/\mu(G),$
where $\mu(G)$ is the \emph{algebraic connectivity} of $G$, namely,
the second smallest eigenvalue of its Laplacian matrix
$L(G)=\Delta-A$. Thus the problem of maximizing relaxation time on
regular graphs is directly related to minimizing algebraic
connectivity.

Aldous and Fill \cite[p.~217]{AldousFill} conjectured that, among all regular graphs on $n$ vertices,
$\max \tau\le (1+o(1))\frac{3n^2}{2\pi^2}$,
with asymptotic equality for even $n$. This is closely related to the problem of minimizing algebraic connectivity under a fixed degree condition. Guiduli and Mohar conjectured that graphs with minimum algebraic connectivity have a path-like structure with prescribed blocks; see \Cref{fig:Mohar}. Abdi and Ghorbani \cite{AbGhDiam} proposed the corresponding conjecture for regular graphs. For odd $d$, the conjectured structure is the same as in the Guiduli--Mohar conjecture. For even $d$, however, a $d$-regular graph cannot contain a bridge, and hence a different structure is required; see \Cref{fig:d-even}.

Part~I~\cite{AbGhCompanion} proves the Aldous--Fill conjecture, the Guiduli--Mohar conjecture, and the regular case for fixed odd degree. An independent proof of the Aldous--Fill conjecture has also been announced by Zhu~\cite{Zhu}. In this paper, we address the remaining case of $d$-regular graphs for fixed even $d$ and, in particular, settle the even-degree case of the Abdi--Ghorbani conjecture.

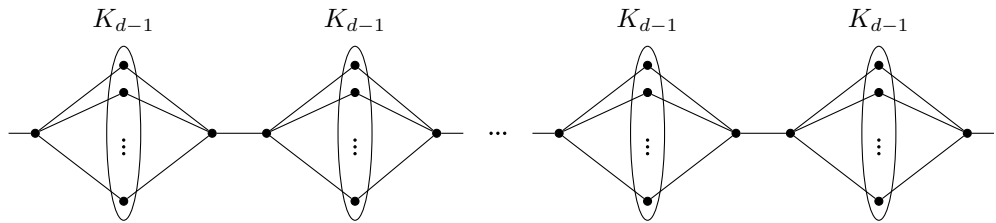
\begin{figure}[tbp]
	\centering
	\begin{tikzpicture}[scale=.9]
		\draw (1,0) ellipse (.25 and 1.28);
		\vertex[fill] (0) at (-.3,0) [] {};
		\vertex[fill] (1) at (1,1) [] {};
		\vertex[fill] (2) at (1,.6) [] {};
		\vertex[fill] (3) at (1,-1) [] {};
		\vertex[fill] (4) at (2.3,0) [] {};
		\vertex[fill] (5) at (3.1,0) [] {};
		\draw (4.4,0) ellipse (.25 and 1.28);
		\vertex[fill] (6) at (4.4,1) [] {};
		\vertex[fill] (7) at (4.4,.6) [] {};
		\vertex[fill] (8) at (4.4,-1) [] {};
		\vertex[fill] (9) at (5.6,0) [] {};
		\vertex[fill] (12) at (7.4,0) [] {};
		\draw (8.7,0) ellipse (.25 and 1.28);
		\vertex[fill] (13) at (8.7,1) [] {};
		\vertex[fill] (14) at (8.7,.6) [] {};
		\vertex[fill] (15) at (8.7,-1) [] {};
		\vertex[fill] (16) at (10,0) [] {};
		\vertex[fill] (17) at (10.8,0) [] {};
		\draw (12.1,0) ellipse (.25 and 1.28);
		\vertex[fill] (18) at (12.1,1) [] {};
		\vertex[fill] (19) at (12.1,.6) [] {};
		\vertex[fill] (20) at (12.1,-1) [] {};
		\vertex[fill] (00) at (13.4,0) [] {};
		\tikzstyle{vertex}=[circle, draw, inner sep=.3pt, minimum size=.3pt]
		\vertex[fill] () at (1,-.1) [] {};
		\vertex[fill] () at (1,-.2) [] {};
		\vertex[fill] () at (1,-.3) [] {};
		\vertex[fill] () at (4.4,-.1) [] {};
		\vertex[fill] () at (4.4,-.2) [] {};
		\vertex[fill] () at (4.4,-.3) [] {};
		\vertex[fill] () at (6.4,0) [] {};
		\vertex[fill] () at (6.5,0) [] {};
		\vertex[fill] () at (6.6,0) [] {};
		\vertex[fill] () at (8.7,-.1) [] {};
		\vertex[fill] () at (8.7,-.2) [] {};
		\vertex[fill] () at (8.7,-.3) [] {};
		\vertex[fill] () at (12.1,-.1) [ ] {};
		\vertex[fill] () at (12.1,-.2) [ ] {};
		\vertex[fill] () at (12.1,-.3) [] {};
		\tikzstyle{vertex}=[circle, draw, inner sep=0pt, minimum size=0pt]
		\vertex[] (s) at (-.7,0) [label=left:\footnotesize{}] {};
		\vertex[] (ss) at (13.8,0) [label=left:\footnotesize{}] {};
		\vertex[] () at (1,1.3) [label=above:\footnotesize{$K_{d-1}$}] {};
		\vertex[] () at (4.4,1.3) [label=above:\footnotesize{$K_{d-1}$}] {};
		\vertex[] (10) at (6.0,0) [label=left:\footnotesize{}] {};
		\vertex[] (11) at (7,0) [label=left:\footnotesize{}] {};
		\vertex[] () at (8.7,1.3) [label=above:\footnotesize{$K_{d-1}$}] {};
		\vertex[] () at (12.1,1.3) [label=above:\footnotesize{$K_{d-1}$}] {};
		\path
		(1) edge (0)
		(2) edge (0)
		(3) edge (0)
		(1) edge (4)
		(2) edge (4)
		(3) edge (4)
		(4) edge (5)
		(6) edge (5)
		(7) edge (5)
		(8) edge (5)
		(6) edge (9)
		(7) edge (9)
		(8) edge (9)
		(9) edge (10)
		(11) edge (12)
		(12) edge (13)
		(12) edge (14)
		(12) edge (15)
		(16) edge (13)
		(16) edge (14)
		(16) edge (15)
		(16) edge (17)
		(17) edge (18)
		(17) edge (19)
		(17) edge (20) 
		(18) edge (00)
		(19) edge (00)
		(20) edge (00)
		(s) edge (0)
		(ss) edge (00);
	\end{tikzpicture}
	\caption{Middle-block structure of $\mu$-minimal graphs with minimum degree at least $d$, and of $\mu$-minimal regular graphs of odd degree $d$, for fixed $d\ge3$ and sufficiently large order. Only a bounded number of vertices at the two ends are omitted. Here $K_l$ is the complete graph of order $l$.}\label{fig:Mohar}
\end{figure}

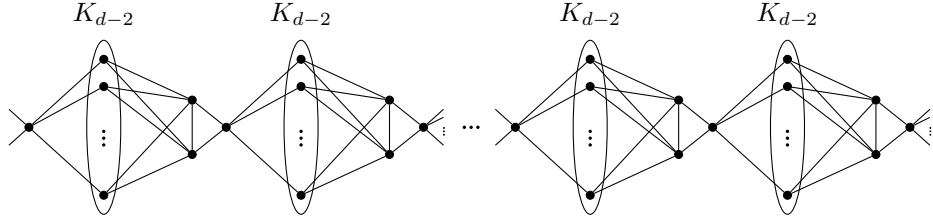
\begin{figure}[tbp]
	\centering
	{\begin{tikzpicture}[scale=.9]
			\draw (-1.9+.2,0) ellipse (.25 and 1.28);
			\vertex[fill] (05) at (-.6+.2,.4) [ ] {};
			\vertex[fill] (04) at (-.6+.2,-.4) [ ] {};
			\vertex[fill] (03) at (-1.9+.2,1) [ ] {};  
			\vertex[fill] (02) at (-1.9+.2,.6) [ ] {};
			\vertex[fill] (01) at (-1.9+.2,-1) [ ] {};
			\vertex[fill] (0) at (-3+.2,0) [ ] {};
			\draw (1+.2,0) ellipse (.25 and 1.28);
			\vertex[fill] (1) at (-.1+.2,0) [ ] {};
			\vertex[fill] (5) at (2.3+.2,.4) [ ] {};
			\vertex[fill] (55) at (2.3+.2,-.4) [ ] {};
			\vertex[fill] (555) at (2.8+.2,0) [ ] {};
			\vertex[fill] (2) at (1+.2,1) [ ] {};
			\vertex[fill] (3) at (1+.2,.6) [ ] {};
			\vertex[fill] (4) at (1+.2,-1) [ ] {};

			\draw (5.7-.25,0) ellipse (.25 and 1.28); 
			\vertex[fill] (6) at (4.6-.25,0) [ ] {};
			\vertex[fill] (07) at (5.7-.25,1) [ ] {};
			\vertex[fill] (8) at (5.7-.25,.6) [ ] {};
			\vertex[fill] (9) at (5.7-.25,-1) [ ] {};
			\vertex[fill] (10) at (7-.25,.4) [ ] {};
			\vertex[fill] (11) at (7-.25,-.4) [ ] {};
			\vertex[fill] (12) at (7.5-.25,0) [ ] {};
			\draw (8.6-.25,0) ellipse (.25 and 1.28);
			\vertex[fill] (13) at (8.6-.25,1) [ ] {};
			\vertex[fill] (14) at (8.6-.25,.6) [ ] {}; 
			\vertex[fill] (15) at (8.6-.25,-1) [ ] {};
			\vertex[fill] (16) at (9.9-.25,.4) [ ] {}; 
			\vertex[fill] (17) at (9.9-.25,-.4) [ ] {};
			\vertex[fill] (18) at (10.4-.25,0) [ ] {};
			\tikzstyle{vertex}=[circle, draw, inner sep=0pt, minimum size=0pt]  
			\vertex[fill] (a) at (10.7-.25,.27) [ ] {}; 
			\vertex[fill] (b) at (10.7-.25,.15) [ ] {}; 
			\vertex[fill] (c) at (10.7-.25,-.27) [ ] {}; 
			\vertex[fill] (055) at (-3.3+.2,.27) [ ] {};
			\vertex[fill] (0555) at (-3.3+.2,-.27) [ ] {};
			\vertex[fill] (7) at (3.1+.2,.27) [ ] {};  
			\vertex[fill] (7a) at (3.1+.2,.15) [ ] {}; 
			\vertex[fill] (77) at (3.1+.2,-.27) [ ] {};
			\vertex[fill] (777) at (4.3-.25,.27) [ ] {};  
			\vertex[fill] (7777) at (4.3-.25,-.27) [ ] {};
			\vertex[fill] () at (1+.2,1.3) [label=above:\footnotesize{$K_{d-2}$}] {};
			\vertex[fill] () at (-1.9+.2,1.3) [label=above:\footnotesize{$K_{d-2}$}] {};
			\vertex[fill] () at (5.7-.25,1.3) [label=above:\footnotesize{$K_{d-2}$}] {};
			\vertex[fill] () at (8.6-.25,1.3) [label=above:\footnotesize{$K_{d-2}$}] {};
			\tikzstyle{vertex}=[circle, draw, inner sep=.1pt, minimum size=.1pt]
			\vertex[fill] () at (3.1+.2,-.01) [ ] {};  
			\vertex[fill] () at (3.1+.2,-.06) [ ] {}; 
			\vertex[fill] () at (3.1+.2,-.11) [ ] {};
			\vertex[fill] () at (10.7-.25,-.01) [ ] {};  
			\vertex[fill] () at (10.7-.25,-.06) [ ] {}; 
			\vertex[fill] () at (10.7-.25,-.11) [ ] {};
			\tikzstyle{vertex}=[circle, draw, inner sep=.3pt, minimum size=.3pt]
			\vertex[fill] () at (1+.2,-.060) [ ] {};
			\vertex[fill] () at (1+.2,-.160) [ ] {};
			\vertex[fill] () at (1+.2,-.260) [ ] {};
			\vertex[fill] () at (-1.9+.2,-.060) [ ] {};
			\vertex[fill] () at (-1.9+.2,-.160) [ ] {};
			\vertex[fill] () at (-1.9+.2,-.260) [ ] {};
			\vertex[fill] () at (5.7-.25,-.060) [ ] {};
			\vertex[fill] () at (5.7-.25,-.160) [ ] {};
			\vertex[fill] () at (5.7-.25,-.260) [ ] {};
			\vertex[fill] () at (8.6-.25,-.060) [ ] {};
			\vertex[fill] () at (8.6-.25,-.160) [ ] {};
			\vertex[fill] () at (8.6-.25,-.260) [ ] {};
			\vertex[fill] () at (3.6,0) [ ] {};  
			\vertex[fill] () at (3.7,0) [ ] {};  
			\vertex[fill] () at (3.8,0) [ ] {};  
			\path
			(18) edge (a)
			(18) edge (b)
			(18) edge (c)
			(0) edge (01)
			(0) edge (055)
			(0) edge (0555)
			(0) edge (02)
			(0) edge (03)
			(04) edge (01)
			(04) edge (02)
			(04) edge (03)
			(05) edge (01)
			(05) edge (02)
			(05) edge (03)
			(04) edge (1)
			(05) edge (1) 
			(05) edge (04)  
			(1) edge (2) 
			(1) edge (3) 
			(1) edge (4) 
			(5) edge (2)
			(5) edge (3)
			(5) edge (4)
			(55) edge (2)
			(55) edge (3)
			(55) edge (4)
			(55) edge (5)
			(555) edge (5)
			(555) edge (55)
			(555) edge (7)
			(555) edge (7a)
			(555) edge (77)
			(6) edge (777)
			(6) edge (7777)
			(6) edge (07)
			(6) edge (8)
			(6) edge (9)
			(10) edge (07)
			(10) edge (8)
			(10) edge (9)
			(11) edge (07)
			(11) edge (8)
			(11) edge (9)
			(11) edge (10)
			(12) edge (10)
			(12) edge (11)
			(12) edge (14)
			(12) edge (15)
			(12) edge (13)
			(16) edge (14)
			(16) edge (15)
			(16) edge (13)
			(17) edge (14)
			(17) edge (15)
			(17) edge (13)
			(17) edge (16)
			(17) edge (18)
			(16) edge (18) ;                     
	\end{tikzpicture}}\caption{
		Middle-block structure of $\mu$-minimal regular graphs of fixed even degree $d\ge4$ and sufficiently large order. Consecutive blocks share a singleton vertex. Only a bounded number of vertices at the two ends are omitted.}\label{fig:d-even}
\end{figure}

\subsection{The main results}
Our first theorem is the even-degree counterpart of the structural
theorem in Part~I. No assumption about diameter or block structure is
made.

\begin{theorem}\label{thm:evenfull}
For every fixed even $d\ge4$, there are constants $N_d,B_d$ such that
every $\mu$-minimal $d$-regular graph $G$ of order $n\ge N_d$ is
path-like. Outside two connected portions at the ends of its
block-tree, containing at most $B_d$ vertices altogether, it is an
uninterrupted chain of $M_d$ blocks. Consecutive blocks share their
singleton endpoints with consistent orientation. In particular,
\begin{equation}\label{eq:even-fulldiameter}
 \dm(G)=\frac{3n}{d+1}+O_d(1).
\end{equation}
\end{theorem}

For $d=4$, we use the characterization of
\cite[Theorem~3.2]{AbGhQuartic}. Our proof for $d\ge6$ is given in
\Cref{sec:even-exact}. The theorem determines the middle blocks and
confines the exceptions to bounded end subgraphs; it does not identify
those end subgraphs individually.

The asymptotic formula for minimum algebraic connectivity follows from the structural result by applying the method of \cite[Theorem~1.8 and Table~1]{AbGhDiam} for estimating the algebraic connectivity of path-like graphs. Let $m_d(n)$ denote the minimum algebraic connectivity among all connected $d$-regular graphs on $n$ vertices.

\begin{corollary}\label{thm:evenminimum}
	For every fixed even $d\ge4$, 
	\begin{equation}\label{eq:even-minimum}
m_d(n)=	\frac{2(d-2)\pi^2}{n^2}+O_d(n^{-3}).
\end{equation}
\end{corollary}

We remark that, for every fixed even $d\ge4$, the maximum diameter
among connected $d$-regular graphs of order $n$ is
\begin{equation}\label{eq:maximum-diameter}
 \frac{3n}{d+1}+O_d(1);
\end{equation}
see \cite{Caccetta,Erdos} and
\cite[Theorems~5.1--5.2]{AbGhDiam}. In
\cite[Conjecture~5.4]{AbGhDiam}, it is conjectured that a graph whose
algebraic connectivity is asymptotically minimum has asymptotically
maximum diameter within its respective family. The following result
settles the even-regular case for every fixed even $d\ge6$.

\begin{theorem}\label{cor:diameter}
Fix an even $d\ge6$. If $G_n$ is $d$-regular of order $n$ and
$\mu(G_n)/m_d(n)\to1$, then
\begin{equation}\label{eq:diameterimplication}
 \dm(G_n)=\frac{3n}{d+1}+o(n).
\end{equation}
\end{theorem}

For $d=4$, \cite[Theorem~3.2]{AbGhQuartic} gives the diameter conclusion
for exact minimizers; it does not establish the corresponding assertion
for graphs whose algebraic connectivity is only asymptotically minimum. Our result concerns asymptotically maximum diameter; it does not assert that an exact minimizer attains the maximum possible diameter for every order.

Our last main result allows the regular degree to grow with $n$. 

\begin{theorem}\label{thm:evenuniform}
For a fixed even integer $r\ge4$, put
\[
 c_r=\min\left\{\frac{2(r-2)}r\pi^2,\,16\right\}.
\]
Then
\begin{equation}\label{eq:even-uniformminimum}
 \lim_{n\to\infty}
 \min_{\substack{|V(G)|=n\\G\text{ is }d\text{-regular}\\
                   d\ge r,\ d\text{ even}}}
 \frac{n^2\mu(G)}d=c_r.
\end{equation}
The minimum is over connected  graphs, and sharpness holds
through all sufficiently large orders.
\end{theorem}

Thus $c_r=2(r-2)\pi^2/r$ for $r=4,6,8,10$, and
$c_r=16$ for every even $r\ge12$.
The proof also gives lower limit $2\pi^2$ for $n^2\mu(G)/d$
when $d\to\infty$ and $d=o(n)$; see \Cref{thm:growing}.
In particular, the maximum relaxation time over even regular degrees
is $(1+o(1))n^2/\pi^2$, and every maximizer at a sufficiently large
order is quartic. Over even regular degrees at least six the maximum
is $(1+o(1))3n^2/(4\pi^2)$; see \Cref{cor:evenrelaxation}.

\subsection{Relation to Part I and earlier work}

L.  Babai (see \cite{Guiduli}) made a conjecture that described the structure of  $\mu$-minimal cubic (i.e. $3$-regular) graphs. Guiduli \cite{Guiduli} (see also \cite{GuiduliThesis}) proved that $\mu$-minimal cubic graphs are path-like, built from specific blocks.
The result of Guiduli was improved  later  by Brand, Guiduli, and Imrich \cite{Imrich}. They completely characterized $\mu$-minimal cubic  graphs and confirmed the Babai conjecture. For every admissible even order, the minimizing cubic graph is unique. (Cubic graphs always have even orders.)
Abdi, Ghorbani and Imrich~\cite{AbGhIm} showed that the algebraic connectivity  of these graphs is $(1+o(1))\frac{2\pi^2}{n^2}$, confirming  the Aldous--Fill conjecture for $d=3$.
Guiduli \cite[Problem~5.2]{GuiduliThesis} asked for a generalization of the aforementioned result of Brand, Guiduli, and Imrich, namely the characterization of $\mu$-minimal $d$-regular graphs. 
In this direction, Abdi and Ghorbani \cite{AbGhQuartic} gave a nearly complete characterization
of $\mu$-minimal quartic  (i.e. $4$-regular) graphs.  Based on that,  they established   the Aldous--Fill conjecture for $d=4$. 
Guiduli and Mohar proposed the following generalization of the cubic
result by considering graphs with minimum degree $d$ rather than
regular graphs.
Part~I~\cite{AbGhCompanion} proves the  characterization in
the two minimum-degree classes $\delta=d$ and $\delta\ge d$, for every
fixed $d\ge3$, and in the regular class for fixed odd $d\ge3$.
Its numerical corollary gives the minimum algebraic connectivity
$$\frac{(d-1)\pi^2}{n^2}+O_d(n^{-3})$$
in each of these classes, with $n$ even in the odd-regular case.
It also proves the corresponding near-minimum diameter implication.
Its uniform minimum-degree theorem gives the sharp constant
$\min\{(r-1)\pi^2/r,8\}$ when $\delta(G)\ge r$, and the
Aldous--Fill bound follows by specialization to regular graphs.

The minimum-degree theorems include even values of $d$, but they do
not determine the minimizers in the smaller even-regular class.
Nor does their constant $8$ replace the even-cut estimate needed for
\Cref{thm:evenuniform}. We prove the even-cut version of its component argument in
\Cref{sec:evenrelaxation}, together with an exact finite-order
bound for the sublinear-degree regime.
All local comparisons, replacement estimates, and growing-degree
arguments are proved here. The quartic characterization is the only
structural theorem imported from another paper. References to Part~I explain the parallel
methods and results, rather than supply unstated proof steps.

\subsection{Outline of the proofs}
Our proofs use Fiedler vectors and changes to a graph that reduce its
algebraic connectivity. These methods were used in
\cite{AbGhMin,AbGhDiam,AbGhIm,AbGhQuartic}, and also in Part~I
\cite{AbGhCompanion}. A Fiedler vector assigns a real number to each
vertex. If these numbers have sum zero and sum of squares one, then
the sum of $(x_u-x_v)^2$ over all edges $uv$ is exactly $\mu(G)$.
We use this formula to compare graphs of the same order and degree.
The quartic case is supplied by \cite{AbGhQuartic}; the structural
argument below is for even $d\ge6$.

First, we construct graphs made of a long chain of $M_d$ blocks and
two end subgraphs whose sizes are bounded in terms of $d$. The method
of comparing chains with paths in \cite{AbGhDiam} gives
\[
 \mu(G_{n,d})\le\frac{2(d-2)\pi^2}{n^2}+O_d(n^{-3}).
\]
Thus a minimizing graph must have algebraic connectivity at most
this value.

Next, we list the vertices in increasing order of their Fiedler-vector
values. We count the edges joining the vertices before each position
to those after it. The degree condition gives inequalities relating
these counts to the changes in the vector values. We apply the
inequalities to short consecutive groups of vertices and then add
them. This compares the changes in the vector values with those along
a path. Equality in consecutive
groups forces those groups to be $M_d$ blocks. The inequalities also
show that a graph close to the minimum is made mostly of such blocks
along one path of its block-tree. For an exact minimizer, at most
$O_d(n^{2/3})$ vertices remain outside the selected blocks.

We now assume that the graph attains the minimum. This allows us to
reduce the number of remaining vertices to a bound depending only on
$d$. If a part between the selected blocks is too
large, replacing it by a graph with the same number of vertices and
the same attachment degrees lowers algebraic connectivity. For a part
at an end, we make the comparison while fixing the vector value at
its attachment vertex. These arguments bound the size of each part
that remains. Swapping a bounded part with a neighboring $M_d$ block
then shows that it must be close to an end of the graph. Finally,
changes of edges that preserve every vertex degree rule out branches
at the ends. This proves the asserted structure.

Once the structure is known, the comparison with a path from
\cite{AbGhDiam} gives the formula for the minimum algebraic
connectivity. The two bounded end subgraphs affect only the error
term. Each further $M_d$ block adds $d+1$ vertices and three to the
distance between the ends of the chain. Counting these blocks gives
the diameter conclusions, both for exact minimizers and for graphs
whose algebraic connectivity is close to the minimum.

Finally, we allow $d$ to grow with $n$. When $d=o(n)$, we extend the
comparison path by repeating its endpoint values. This gives a lower
bound with an error that can still be controlled as $d$ grows. When
$d$ is proportional to $n$, we divide the vertices into a bounded
number of large connected sets on which the Fiedler-vector values
are nearly equal. We compare the average values on these sets. Even
regular degree forces at least two edges across every cut, and this
gives the lower bound $16$ for the limiting value of $n^2\mu(G)/d$.
A regular graph built from two large parts joined by two edges
attains this limit. Together, the two cases give the uniform bound
and the results on relaxation time.

\subsection{Organization}
\Cref{sec:prelim} introduces the notation and the facts about Fiedler
vectors and paths used in the proofs. \Cref{sec:local} proves the
inequalities for short groups of ordered vertices and constructs
graphs for comparison. \Cref{sec:evenstructure} shows that graphs
with algebraic connectivity close to the minimum consist mostly of
$M_d$ blocks. \Cref{sec:even-exact} proves the stated structure of
exact minimizers. \Cref{sec:chainvalue} gives the formula for the
minimum algebraic connectivity, and \Cref{sec:diameter} proves the
diameter results. \Cref{sec:evenrelaxation} allows the degree to vary
with the order and proves the bounds on relaxation time.

\section{Fiedler vectors and paths}\label{sec:prelim}
The \emph{order} of a graph is its number of vertices. We write
$\deg_G(v)$ for the degree of a vertex $v$, $\delta(G)$ for the
minimum degree, and $\dm(G)$ for the diameter. A $d$-regular graph
is \emph{$\mu$-minimal} if its algebraic connectivity equals the
minimum $m_d(n)$ defined in the introduction, where $n$ is its order.

A \emph{cut vertex} is a vertex whose deletion disconnects the graph,
and a \emph{bridge} is an edge whose deletion disconnects it. A
\emph{block} is a maximal connected subgraph without a cut vertex;
in particular, a bridge forms a block on two vertices. The
\emph{block-tree} has one node for each block and one for each cut
vertex, with a block joined to the cut vertices it contains. Following
\cite{AbGhDiam}, a graph with at least two blocks is \emph{path-like}
when this tree is a path. Its two end blocks are the blocks at the
ends of that path. An \emph{end subgraph} may contain several blocks
near one end.

Write $K_t$ for the complete graph on $t$ vertices. For disjoint
graphs $H_1,\ldots,H_t$, the \emph{sequential join}
$H_1+\cdots+H_t$ adds all edges between consecutive graphs and no
edges between nonconsecutive graphs. Put
\[
 L_d=K_{d+1}-e,\qquad M_d=K_1+K_2+K_{d-2}+K_1,
\]
where $e$ is one edge of $K_{d+1}$. In an $L_d$ chain, consecutive
copies are disjoint and joined by bridges at endpoints of their
missing edges; see \Cref{fig:Mohar}. In an $M_d$ chain, consecutive
copies share their outer singleton vertices, with all copies oriented
in the same direction; see \Cref{fig:d-even}. Each $M_d$ block has
$d+2$ vertices. Every additional block adds $d+1$ vertices and three
to the distance between the outer singletons of the chain. Its
internal vertices have degree $d$, and each shared singleton has
degree $2+(d-2)=d$.

For a vertex set $S$, let $e(S)$ be the number of edges with both
endpoints in $S$. In a $d$-regular graph, the number of edges with
exactly one endpoint in $S$ is $d|S|-2e(S)$. Thus, when $d$ is even,
every nonempty proper vertex set has a positive even number of edges
leaving it, since the graph is connected. In particular, an
even-regular graph has no bridge.

Except in the last section, $d$ is fixed. The constant in
$O_d(\cdot)$ may depend on $d$, but not on $n$. Constants with
subscript $d$ may be increased during a proof. All vectors are real,
their norms are Euclidean, and $\1$ denotes the all-one vector.
An ordered vector is nondecreasing; equal entries are allowed.
For even $d$, there is no parity restriction on sufficiently large
orders $n$.

An eigenvector corresponding to $\mu(G)$ is a {\em Fiedler vector}.
For a nonzero real vector $\x=(x_1,\ldots,x_n)^\top$, the quantity
$\x^\top L(G)\x/\|\x\|^2$ is a {\em Rayleigh quotient}.
It is well known that
\begin{align}
 \mu(G)&=\min_{\x\ne\0,\,\x\perp\1}
       \frac{\x^\top L(G)\x}{\|\x\|^2},\label{eq:Rayleigh}\\
 \x^\top L(G)\x&=\sum_{ij\in E(G)}(x_i-x_j)^2.\label{eq:quadratic}
\end{align}
At a vertex $i$ of degree $d_i$, a Fiedler vector satisfies the
{\em eigen-equation}
\begin{equation}\label{eq:eigenequation}
 \mu x_i=d_ix_i-\sum_{j:\,ij\in E(G)}x_j
        =\sum_{j:\,ij\in E(G)}(x_i-x_j).
\end{equation}
For $\bar y=n^{-1}\sum_i y_i$, centering leaves the quadratic form
unchanged and gives
\begin{equation}\label{eq:centering}
 \|\y-\bar y\1\|^2=\sum_i y_i^2-\frac1n\left(\sum_i y_i\right)^2.
\end{equation}
These facts are recalled in \cite[Section~2]{AbGhDiam}.
We also write $\E_H(\z)=\z^\top L(H)\z$ for the quadratic form on a
subgraph $H$. For every nonconstant vector $\y$, centering and
\eqref{eq:Rayleigh} give
\begin{equation}\label{eq:centeredRayleigh}
 \mu(G)\le\frac{\y^\top L(G)\y}{\|\y-\bar y\1\|^2}.
\end{equation}
In particular, if $\x$ is a mean-zero unit Fiedler vector of $G$,
then a vector on a graph $G'$ of the same order proves $\mu(G')<\mu(G)$
whenever
\begin{equation}\label{eq:strictcomparison}
 \y^\top L(G')\y-\mu(G)\|\y\|^2
       +\frac{\mu(G)}n\left(\sum_i y_i\right)^2<0.
\end{equation}
If an unchanged vector has Rayleigh quotient $\mu(G)$ on $G'$ but
no longer satisfies its eigen-equations, equality is again impossible,
and $\mu(G')<\mu(G)$.

For $P_n$, the path graph on $n$ vertices, we know that
(see \cite{fiedler1973algebraic})
\[
 p_n:=\mu(P_n)=2(1-\cos(\pi/n)),
\]
and by \cite[Theorem~5.6.1]{Spielman}, its increasing unit Fiedler vector has
components
\begin{equation}\label{eq:pathvector}
 v_i=-\sqrt{\frac2n}\cos\frac{(2i-1)\pi}{2n}
\end{equation}
for $i=1,\ldots,n$. Also define $h_i=v_{i+1}-v_i$ for
$i=1,\ldots,n-1$.
From \eqref{eq:Rayleigh} and \eqref{eq:centering} it follows that
\begin{equation}\label{eq:pathineq}
 \y^\top L(P_n)\y=\sum_{i=1}^{n-1}(y_{i+1}-y_i)^2
       \ge p_n\|\y-\bar y\1\|^2.
\end{equation}
We use the same near-equality lemma as in \cite{AbGhCompanion}
and include its proof.
\begin{lemma}\label{lem:pathstable}
	Let $n\ge3$ and $\theta\ge0$. Suppose there is a constant $C>0$ such that
	\begin{align*}
 \left|\|\y-\bar y\1\|^2-1\right|&\le C\theta,\\
 0\le\y^\top L(P_n)\y-p_n\|\y-\bar y\1\|^2&\le C\theta/n^2.
\end{align*}
	Then there is a constant $C'>0$, depending only on $C$, such that,
	for sufficiently small $\theta$, after changing the sign of $\y$ if
	necessary,
	\begin{equation}\label{eq:pathclose}
		\sum_{i=1}^{n-1}(y_{i+1}-y_i-h_i)^2
		\le C'\theta/n^2.
	\end{equation}
	No sign change is needed when $\y$ is nondecreasing. If
	$A\subseteq\{1,\ldots,n-1\}$ satisfies
	\begin{equation}\label{eq:smallset}
		\sum_{i\in A}(y_{i+1}-y_i)^2\le C\theta/n^2,
	\end{equation}
	then, after increasing $C'$ if necessary,
	\[
	|A|\le C'\theta^{1/3}n.
	\]
\end{lemma}

\begin{proof}
Put $L=L(P_n)$. The next positive eigenvalue after $p_n$ is
$\widetilde p_n=2(1-\cos(2\pi/n))$, so, for $n\ge3$,
\begin{equation}\label{eq:pathgap1}
 \frac{\widetilde p_n}{p_n}
 =4\cos^2\frac{\pi}{2n}\ge3.
\end{equation}
We first prove the following estimate for a mean-zero unit vector $\z$,
after choosing its sign so that $\langle\z,\bv\rangle\ge0$:
\begin{equation}\label{eq:pathunitstable}
 (\z-\bv)^\top L(\z-\bv)
 \le2\bigl(\z^\top L\z-p_n\bigr).
\end{equation}
Write $\z=c\bv+\bu$, where $\bu\perp\1,\bv$ and $c\ge0$.
Then $c^2+\|\bu\|^2=1$. Every eigenvalue contributing to $\bu$
is at least $3p_n$, and $\lambda+p_n\le2(\lambda-p_n)$ for
$\lambda\ge3p_n$. Hence
\[
 \bu^\top L\bu+p_n\|\bu\|^2
 \le2\bigl(\bu^\top L\bu-p_n\|\bu\|^2\bigr).
\]
Since $(1-c)^2\le1-c^2=\|\bu\|^2$, we obtain
\[
 (\z-\bv)^\top L(\z-\bv)
 =p_n(1-c)^2+\bu^\top L\bu
 \le2\bigl(\z^\top L\z-p_n\bigr),
\]
which proves \eqref{eq:pathunitstable}.

Now put $s=\|\y-\bar y\1\|$. We may assume $\theta\le1$ and
$C\theta\le1/2$, so $s>0$. Change the sign of $\y$ if necessary
so that $\langle\y,\bv\rangle\ge0$, and apply
\eqref{eq:pathunitstable} to $\z=(\y-\bar y\1)/s$.
The inequality $(u+v)^2\le2u^2+2v^2$, applied to each path difference,
and $\y-\bar y\1-\bv=s(\z-\bv)+(s-1)\bv$ give
\begin{align*}
 (\y-\bar y\1-\bv)^\top L(\y-\bar y\1-\bv)
 &\le2s^2(\z-\bv)^\top L(\z-\bv)+2p_n(s-1)^2\\
 &\le4\bigl(\y^\top L\y-p_ns^2\bigr)+2p_n(s-1)^2\\
 &\le\frac{(4C+2\pi^2C^2)\theta}{n^2}.
\end{align*}
For the last step we used
$|s-1|=|s^2-1|/(s+1)\le C\theta$ and $p_n\le\pi^2/n^2$.
Adding a constant does not change path differences, so this proves
\eqref{eq:pathclose}. If $\y$ is nondecreasing, then
\[
 \langle\y,\bv\rangle
 =\frac1n\sum_{i<j}(y_j-y_i)(v_j-v_i)>0,
\]
because $\bv$ is strictly increasing and $\y$ is not constant.
Thus no sign change is needed.

It remains to prove the last assertion. From
	\eqref{eq:pathvector},
	\[
	h_i
	=
	2\sqrt{\frac2n}
	\sin\frac{\pi}{2n}
	\sin\frac{\pi i}{n}.
	\]
	Put
	\[
	r_i=\min\{i,n-i\}.
	\]
	Using $\sin t\ge2t/\pi$ on $[0,\pi/2]$, we obtain
	$\sin\frac{\pi}{2n}\ge\frac1n,$ and 	$\sin\frac{\pi i}{n}
	=	\sin\frac{\pi r_i}{n}
	\ge\frac{2r_i}{n}.$
	Consequently
	\begin{equation}\label{eq:hlower}
		h_i\ge
		\frac{4\sqrt2\,r_i}{n^{5/2}}.
	\end{equation}

	Let $|A|=k$. Each positive integer occurs at most twice among the
	numbers $r_i$. Hence, if their values on $A$ are arranged increasingly,
	the $j$th one is at least $j/2$. By \eqref{eq:hlower},
	\begin{equation}\label{eq:cubicmass}
		\sum_{i\in A}h_i^2
		\ge \frac{8}{n^5}\sum_{j=1}^k j^2
		\ge \frac{8}{3}\frac{k^3}{n^5}.
	\end{equation}
	On the other hand, \eqref{eq:pathclose} and
	\eqref{eq:smallset} give
	\[
	\sum_{i\in A}h_i^2
	\le
	2\sum_{i\in A}(y_{i+1}-y_i)^2
	+2\sum_{i=1}^{n-1}(y_{i+1}-y_i-h_i)^2
	\le
	\frac{2(C+C')\theta}{n^2}.
	\]
	Together with \eqref{eq:cubicmass}, this gives
	\[
	k^3\le\frac34(C+C')\theta n^3.
	\]
	Thus
	\[
	|A|
	\le
	\left(\frac34(C+C')\right)^{1/3}
	\theta^{1/3}n.
	\]
	Increasing $C'$ once more proves the assertion.
\end{proof}

\begin{lemma}\label{lem:smallentries}
For a unit Fiedler vector $\x$ on a graph of order $n$,
\begin{align}
 \max_i|x_i|&\le\sqrt{(n-1)\mu},\label{eq:edgeestimate}\\
 |x_u-x_v|&\le\mu\sqrt n/2\quad(uv\in E(G)).\notag
\end{align}
In particular, if $\mu=O_d(n^{-2})$, these are
$O_d(n^{-1/2})$ and $O_d(n^{-3/2})$, respectively.
\end{lemma}
\begin{proof}
Join the maximum and minimum components by a path of length at most
$n-1$. Cauchy--Schwarz bounds their squared difference by $(n-1)\mu$;
since $\x\perp\1$, this proves the first assertion.
For the second, order $x_1\le\cdots\le x_n$ and sum
\eqref{eq:eigenequation} over the first $i$ vertices:
\begin{equation}\label{eq:prefixsum}
 \sum_{uv\in E(G),\,u\le i<v}(x_v-x_u)
       =-\mu\sum_{j\le i}x_j.
\end{equation}
If $T=\sum_{j\le i}x_j=-\sum_{j>i}x_j$, then
$T^2/i+T^2/(n-i)\le1$, so $|T|\le\sqrt{i(n-i)/n}\le\sqrt n/2$.
Every nonzero edge difference is one of the nonnegative terms on the
left of some such cut. This proves its stated bound.
\end{proof}

\section{Local estimates and comparison graphs}\label{sec:local}
Throughout this section, $d\ge6$ is even, $D=d+1$, and
$\beta=2(d-2)$. We first explain the coefficient $\beta$, then prove
local comparisons in an arbitrary Fiedler ordering. The same partition
will be used for stability and for replacements with fixed boundary
values. The comparison graphs are constructed at the end of the section.

\subsection{The quadratic form of a standard block}
The following two identities record both a lower bound and the
nonnegative terms lost in taking it. The block calculation in this
subsection also holds for $d=4$; the restriction $d\ge6$ is needed
for the subsequent interval comparisons.
\begin{samepage}
\begin{lemma}\label{lem:comparison-averaging}
For disjoint nonempty vertex sets $A,B$, write $\bar z_A,\bar z_B$
for their average values. If all edges between the two sets are present,
their contribution is
\begin{align}\label{eq:clique-averaging}
 \sum_{u\in A,\,v\in B}(z_u-z_v)^2
 ={}&|A||B|(\bar z_A-\bar z_B)^2\notag\\
 &+|B|\sum_{u\in A}(z_u-\bar z_A)^2
  +|A|\sum_{v\in B}(z_v-\bar z_B)^2.
\end{align}
Consequently, averaging within a clique of vertices with the same
outside neighbors cannot increase the quadratic form. The two
remainder sums vanish precisely when the values are constant on
each of the two sets.
\end{lemma}
\end{samepage}
\begin{proof}
Expand $z_u-z_v$ about the two averages. The mixed sums vanish,
giving \eqref{eq:clique-averaging}. Its two remainder sums are
nonnegative and vanish precisely under the stated conditions.
\end{proof}

\begin{lemma}\label{lem:comparison-weighted-squares}
For positive numbers $c_1,\ldots,c_m$ and real numbers
$g_1,\ldots,g_m$, put $R=\sum_jc_j^{-1}$ and $\Delta=\sum_jg_j$.
Then
\begin{equation}\label{eq:weighted-squares}
 \sum_{j=1}^m c_jg_j^2-\frac{\Delta^2}{R}
 =\sum_{j=1}^m c_j\left(g_j-\frac{\Delta}{Rc_j}\right)^2.
\end{equation}
Thus $\sum_jc_jg_j^2\ge\Delta^2/R$, with equality precisely when
$g_j=\Delta/(Rc_j)$ for every $j$.
\end{lemma}
\begin{proof}
Expanding the right-hand side and using $R=\sum_jc_j^{-1}$ and
$\Delta=\sum_jg_j$ gives \eqref{eq:weighted-squares}. The equality
condition follows because all the weights are positive.
\end{proof}
This is the weighted Cauchy--Schwarz inequality, with its nonnegative
remainder retained.

\begin{lemma}\label{lem:comparison-blockminimum}
For singleton values $x,x'$, give every vertex in the $K_2$ group
value $x+(d-2)(x'-x)/D$, and every vertex in the $K_{d-2}$ group
value $x+(d-1)(x'-x)/D$. The three sets of connecting edges have sizes
$2$, $2(d-2)$ and $d-2$, so the quadratic form of one block is
\begin{equation}\label{eq:comparisonperiod}
 \frac{\beta}{D}(x-x')^2.
\end{equation}
These are the minimizing internal values for fixed $x,x'$.
\end{lemma}
\begin{proof}
By \eqref{eq:clique-averaging}, we may first average on each middle
clique. Then \eqref{eq:weighted-squares}, with weights
$2,2(d-2),d-2$, gives the minimum and the displayed values, since
\[
 R=\frac12+\frac1{2(d-2)}+\frac1{d-2}=\frac D\beta.
\]
\end{proof}
This is the Dirichlet principle in this particular network: the
harmonic extension minimizes the quadratic form with boundary values
fixed; see \cite[Exercise~1.3.11]{DoyleSnell}. It is also the clique
calculation used in \cite[proof of Theorem~3.8]{AbGhDiam}.
\subsection{Cuts and sums of the Fiedler equations}\label{sec:even-cuts}
Write $D=d+1$, $\beta=2(d-2)$, and
\begin{align*}
 \phi(k)&=\frac12\left(\frac1k-\frac1{d-k}\right),\\
 F(k)&=\frac12\left(\frac1k+\frac1{d-k}\right).
\end{align*}
The boundary correction is antisymmetric:
$\phi(d-k)=-\phi(k)$. Also
\begin{equation}\label{eq:even-F}
 F(k)=\frac{d}{2k(d-k)}\le F(2)=\frac{d}{4(d-2)}
 \quad(2\le k\le d-2),
\end{equation}
since $k(d-k)-2(d-2)=(k-2)(d-k-2)\ge0$.
Order a vector so that $x_1\le\cdots\le x_n$, and put
$a_i=x_{i+1}-x_i\ge0$. Let $q_i$ count the edges crossing the cut after
vertex $i$, and put $q_0=q_n=0$. Let $K_{ij}$ count the edges
crossing both cuts $i,j$. Thus, for $i\le j$, it counts the edges
$uv$ with $u\le i$ and $v>j$.

\begin{lemma}\label{lem:even-cut-counts}
For every nondecreasing real vector on a $d$-regular graph,
\begin{equation}\label{eq:even-energy}
 \x^\top L(G)\x=\sum_{i,j=1}^{n-1}K_{ij}a_i a_j.
\end{equation}
The cut counts satisfy the following inequalities, together with
their reversed versions, whenever the displayed cut indices lie
between $0$ and $n$:
\begin{align}
 q_c+q_{c+t}&\ge t(D-t) &&(1\le t\le d),\label{eq:even-cut}\\
 K_{c+s,c+t}&\ge s(D-t)-q_c &&(1\le s\le t\le d),\label{eq:even-overlap}\\
 K_{ij}&\ge\frac{q_i+q_j-d(j-i)}2 &&(i<j).\label{eq:even-two}
\end{align}
In particular, $2K_{i,i+1}=q_i+q_{i+1}-d$.
\end{lemma}
\begin{proof}
For an edge $uv$ with $u<v$, we have
$x_v-x_u=\sum_{i=u}^{v-1}a_i$. Expanding its square and summing
proves \eqref{eq:even-energy}.
For \eqref{eq:even-cut}, the $t$ vertices between cuts $c$ and
$c+t$ have total degree $dt$. At most $t(t-1)$ of this total comes
from internal edges, so at least $t(D-t)$ edges leave this set.
They are counted in $q_c+q_{c+t}$, proving the inequality.
For \eqref{eq:even-overlap}, consider the first $s$ vertices after
cut $c$. Of their total degree $ds$, at most $s(s-1)$ comes from
internal edges, at most $s(t-s)$ from edges to the next $t-s$
vertices, and at most $q_c$ from edges to the left. At least
$s(D-t)-q_c$ edges therefore go beyond cut $c+t$; all cross both
cuts $c+s,c+t$. For the last inequality, if
$T=\{i+1,\ldots,j\}$, then
$q_i+q_j=d|T|-2e(T)+2K_{ij}$. This gives \eqref{eq:even-two}.
When $j=i+1$, the set $T$ is a singleton and $e(T)=0$, which proves
the stated identity for adjacent cuts.
\end{proof}

A cut is called small when $q_i<d$. Every cut is even, so its small
size belongs to $\{2,4,\ldots,d-2\}$. By \eqref{eq:even-cut}, small
cuts occur singly or in adjacent pairs: two of them cannot be at
distance $2,\ldots,d-1$. At the last cut of a group, record its type
as $(h,k)$ for an adjacent pair of sizes $h,k$, and as $(*,k)$ for
a single cut. Necessarily $h+k\ge d$. An additional position in a
run of high cuts has type $\circ$, and its $\phi$-value is defined
to be zero.

If consecutive recorded positions are $a$ and $a+r$, their interval
is $I=\{a,\ldots,a+r-1\}$, with endpoint values $x_a,x_{a+r}$.
The cut at the right endpoint is not itself in $I$; at a paired right
end the preceding small cut is in $I$. Write $k$
for the small size at the left, when there is one. At the right there
is either a single cut of size $l$, a pair of sizes $g,l$, or a free
end. All other cuts of the interval are high. The desired denominator
is
\begin{equation}\label{eq:even-T}
 T=\frac r\beta+\phi(k)-\phi(l),
\end{equation}
with the free-end terms omitted. When both ends are small, $r\ge d$.
If the right end is a pair, $r=d$ is impossible, since
\eqref{eq:even-cut} at distance $d-1$ would give
$k+g\ge2(d-1)$.

For an interval $I$ write
\begin{align*}
 \Delta_I&=\sum_{i\in I}a_i,\\
 Q_I&=\sum_{i,j\in I}K_{ij}a_i a_j.
\end{align*}
Here and below an interval consists of difference indices; its length
is their number.

For any ordered vector, put
\[
 c_i=\sum_jK_{ij}a_j
    =\sum_{uv\in E(G),\,u\le i<v}(x_v-x_u).
\]
Thus $c_i$ is the sum of the component differences across cut $i$,
and $c_i\ge0$. Now let $\x$ be a mean-zero unit Fiedler vector, and,
as in the companion paper, put
\begin{align*}
 f_i&=-\sum_{v=1}^i x_v\quad(1\le i<n),\\
 f_0&=f_n=0.
\end{align*}
\begin{lemma}\label{lem:even-prefix-identities}
For the mean-zero unit Fiedler vector above, $f_i>0$ for $1\le i<n$,
and
\begin{align}
 c_i&=\mu(G)f_i,\notag\\
 c_{i+1}-c_i&=-\mu(G)x_{i+1}.\label{eq:even-b}
\end{align}
\end{lemma}
\begin{proof}
Indeed,
\[
 f_i=\frac1n\sum_{v\le i<w}(x_w-x_v)>0:
\]
all summands are nonnegative, and equality would force the nonconstant
vector $\x$ to be constant. Summing its eigen-equations gives the first identity in
\eqref{eq:even-b}. Subtracting consecutive identities and using
$f_{i+1}-f_i=-x_{i+1}$ gives the second.
\end{proof}

\begin{lemma}\label{lem:even-local-methods}
For an interval $I$ and a nondecreasing real vector, either of the
following implications bounds $\Delta_I$.
If $R>0$ and $Q_I\ge\Delta_I^2/R$, then
\begin{equation}\label{eq:even-method1}
 \Delta_I\le R\max_{i\in I}c_i.
\end{equation}
If weights satisfy
\begin{align*}
 w_j&\ge0\quad(j\in I),\\
 \sum_{j\in I}K_{ij}w_j&\ge1\quad(i\in I),
\end{align*}
then
\begin{equation}\label{eq:even-method2}
 \Delta_I\le\left(\sum_{j\in I}w_j\right)\max_{i\in I}c_i.
\end{equation}
\end{lemma}
\begin{proof}
For \eqref{eq:even-method1}, use
$Q_I\le\sum_{i\in I}a_i c_i$ and cancel $\Delta_I$; the zero case
is immediate. For \eqref{eq:even-method2}, multiply each displayed row
inequality by $a_i$, sum, and use symmetry of $K$ and
$\sum_{i\in I}K_{ij}a_i\le c_j$.
\end{proof}

We prove that one of these methods supplies $R\le T$ on every
interval used in the final partition. Apart from the stated equality
types, it supplies $R<T$. All interval lengths are bounded in terms
of $d$; thus the strict improvements can subsequently be replaced by
one positive constant depending on $d$.

\subsection{Initial weights for a longer interval}
We record once the alternating weights used in the longer-interval
proofs. The tail of each construction only has to cover its own rows
and supply the missing half in the last row below.

\begin{lemma}\label{lem:even-initial-weights}
Suppose $q_0=k<d$ and cuts $1,\ldots,t$ are high, where $t\ge1$.
Put
\[
 w_0=\frac1k,\qquad
 w_i=\frac1d\quad(1\le i<t,\ i\equiv t-1\pmod2),
\]
and put the other weights zero. Rows $0,\ldots,t-1$ are covered,
meaning that their weighted cut sums are at least one. Row $t$
receives at least $1/2$. The total weight is
\begin{equation}\label{eq:even-initial-four}
 \frac1k+\frac{\lfloor t/2\rfloor}{d}
 \le\frac1k+\frac{t}{2d}.
\end{equation}
When $t$ is even, one may decrease $w_0$ by $1/(2d)$ and keep all
these conclusions. With this decrease, and with the original weights
when $t$ is odd, the total is exactly
\begin{equation}\label{eq:even-prefix-trimmed}
 \frac1k+\frac{t-1}{2d}.
\end{equation}
\end{lemma}
\begin{proof}
The adjacent-cut identity gives $K_{01}\ge k/2$ and
$K_{i,i+1}\ge d/2$ between high cuts. Row $0$ is covered by its
own weight. Each selected high row is covered by its diagonal,
while an unselected row between $1$ and $t-1$ has a selected
neighbor on each side, contributing at least $1/2$ each. At the
left boundary the weight $w_0$ gives the same half contribution.
Row $t$ receives its half from $w_{t-1}$, or from $w_0$ when $t=1$.
If $t$ is even, row $1$ is selected. It covers itself, and its
contribution $K_{01}/d\ge k/(2d)$ permits the stated decrease in
$w_0$. No other row uses the decreased weight in its check.
Counting the selected indices proves both sums.
\end{proof}

\subsection{Short intervals ending at a separating vertex}
A vertex $a$ separates its earlier and later vertices precisely when
$q_{a-1}+q_a=d$: the equality
$2K_{a-1,a}=q_{a-1}+q_a-d$ says that no edge bypasses it.
\begin{lemma}\label{lem:even-separator-counts}
If a separating vertex $a$ has $p$ neighbors to its right, then
\begin{equation}\label{eq:even-J}
 K_{a+s,a+t}\ge
 J_p(s,t):=s(D-t)-\min(p,s)+\max(0,p-t),
 \quad0\le s\le t\le d.
\end{equation}
The numbers $J_p(s,t)$ are exactly the cut counts of
$K_1+K_p+K_{d-p}+K_1$ in its natural order.
\end{lemma}
\begin{proof}
The $s$ vertices after $a$ have no earlier neighbor except
possibly $a$. At most $\min(p,s)$ of their edges meet it. In addition,
at least $\max(0,p-t)$ edges from $a$ go beyond position $a+t$.
These two disjoint counts give \eqref{eq:even-J}.
Counting in the sequential join gives equality in those counts,
which proves the last assertion.
\end{proof}

\subsubsection{The shortest interval and its equality cases}
We use the two endpoint types
$$ P=(2,d-2),\quad Q=(d-2,2).$$
\begin{lemma}\label{lem:even-shortest-separator}
Suppose that $r=D$ and the right type is $(g,p)$, where $p+g=d$.
The local comparison has a denominator $R\le T$. Equality for
nonzero nonnegative differences is possible only if the right type
is $P$ or $Q$. If the left type also belongs to $\{P,Q\}$, the
two types must agree. Every other case is strict.
\end{lemma}
\begin{proof}
The reversed form of \eqref{eq:even-J} compares our interval with
$K_1+K_p+K_g+K_1$. The left size $k$ satisfies $k\ge p$ by
\eqref{eq:even-cut}. Put $\delta=k-p$.
For endpoint values $0,\Delta$ and prescribed first internal value
$u$, the exact minimum on this comparison graph is
\[
 \frac{pg}{D}\Delta^2+\frac D2\left(u-\frac{g}{D}\Delta\right)^2.
\]
To check this, its internal matrix is $DI-\1\1^\top$, whose inverse
is $(I+\1\1^\top)/D$. Equivalently, complete squares after subtracting
the minimizing internal values $g\Delta/D$ and $(g+1)\Delta/D$.
Since the actual first diagonal is $k$, we can add $\delta u^2$.
Thus
\begin{align}
 Q_I&\ge\frac{\Delta^2}{R(k)},\label{eq:even-Rsep}\\
 R(k)&=\frac{D(D+2\delta)}{g\{pD+(d+p)\delta\}}.\notag
\end{align}

For $p\ge4$, a slightly larger denominator avoids any polynomial
expansion. Write
\[
 H=pD+(d+p)(k-p)=Dk+(p-1)(k-p).
\]
Since $k\ge p$, we have $H\ge Dk$, and hence
\begin{equation}\label{eq:even-separator-simple}
 R(k)\le B(k):=\frac{D+2(k-p)}{gk}.
\end{equation}
It suffices to show $B(k)\le T$. To do so, observe that
\[
 B(k)-\phi(k)=\frac2g+\frac{g-2p+2}{2gk}
                         +\frac1{2(d-k)}.
\]
If $g-2p+2\ge0$, this is a convex function of $k$, being a
nonnegative sum of reciprocal affine functions and a constant.
If $g-2p+2<0$, both nonconstant terms are increasing in $k$.
In either case its maximum on $[p,d-2]$ is at an endpoint.
The two endpoint comparisons are
\begin{align}
 (T-B)\big|_{k=p}
 &=\frac{D(p-2)(g-2)}{\beta pg},\notag\\
 (T-B)\big|_{k=d-2}
 &=\frac{(g-2)\{g(p-2)+p(p-4)\}}
          {4gp(d-2)}.\label{eq:even-sepgap}
\end{align}
Every factor has the required sign for $p\ge4$ and $g\ge2$.
Both comparisons are strict if $g>2$. If $g=2$, the interval
$[p,d-2]$ consists only of $k=p=d-2$. This proves exactly the
same equality restriction as the sharper denominator, with no
higher-degree calculation.

When $p=2$, keep $K_{01}\ge k/2$, which follows because cut $1$ is
high. Besides the $J_2$ form, this retains
$(k-2)(a_0^2+a_0a_1)$. Write the two first clique values as $A-v,A+v$
and the other clique values as $B+w_j$, with $\sum w_j=0$. The
retained sum is at least
\[
 kA^2+2g(B-A)^2+g(\Delta-B)^2
              +(2d+4-k)v^2+D\sum_jw_j^2.
\]
Consequently we may take $R=1/k+3/(2(d-2))$, and
\begin{equation}\label{eq:even-J2gap}
 T-R=\frac{d\{k(d-k)-2(d-2)\}}{4(d-2)k(d-k)}\ge0.
\end{equation}
Only $k=2,d-2$ can give equality. For $k=d-2$, equality would force
$a_0=2\Delta/5$, $a_2=\Delta/5$, $a_d=2\Delta/5$, and all other
differences zero. But \eqref{eq:even-two} also gives
$K_{02}\ge(d-6)/2$, whereas $J_2(0,2)=0$. This gives strictness
for $d\ge8$. When $d=6$ and the left type is $(2,4)$,
\eqref{eq:even-J} at the left gives $K_{02}\ge2$, again strict.

We have proved the required comparison on every such shortest
interval.
Every possible equality case ends in $P$ or $Q$. If its left type
is also $P$ or $Q$, the two types are equal. The degree-six cases
just considered satisfy the same restriction. In every other case
the comparison is strict on nonzero nonnegative differences; by
compactness it has a fixed positive improvement for each fixed $d$.
\end{proof}

\subsubsection{All the longer intervals ending at a separator}
\begin{lemma}\label{lem:even-longer-separator}
Let $r=D+t$, where $t\ge1$, with right type $(g,p)$ and $p+g=d$.
Then \eqref{eq:even-method2} gives a denominator $R<T$.
\end{lemma}
\begin{proof}
We give the weights in \eqref{eq:even-method2} explicitly.
Use relative indices $0,\ldots,t+d$. Take the initial weights
of \Cref{lem:even-initial-weights} with total
\eqref{eq:even-prefix-trimmed}. Finally put
\begin{align}
 w_{t+1}&=A:=\frac1{d+p-2},\label{eq:even-pairweights}\\
 w_{t+p}&=B:=\frac{1/g-A}{p},\notag\\
 w_{t+d}&=\frac1g.\notag
\end{align}
All weights are nonnegative. The initial rows and the first half
of row $t$ are covered by \Cref{lem:even-initial-weights}.
The second half of row $t$ comes from $A$, since
$q_{t+1}\ge d+p-2$ by \eqref{eq:even-cut} and hence
$K_{t,t+1}\ge(d+p-2)/2$.

It remains to check the tail. For $1\le s\le p$, the contribution
of the three tail weights is
\[
 A(d+p-2s)+Bsg.
\]
This is linear in $s$, with values $1+Bg\ge1$ at $s=1$ and $1$
at $s=p$. For $p\le s\le d$, it is
$(d-s)(A+pB)+(s-p)/g=1$. These are entries of $J_p$ in
\eqref{eq:even-J}. Thus every row is covered.

The sum of the weights is
\[
 R=\frac1k+\frac1g+\frac{t-1}{2d}
       +\frac1{pg}+\frac{p-1}{p(d+p-2)}.
\]
The last two terms have sum less than $1/(d-2)$: after multiplication
by the positive denominator, the difference has numerator
\[
 p^2g-2(d-2)(p-1)
 =g(p^2-2p+2)-2(p-1)(p-2)\ge2p>0.
\]
On the other hand, \eqref{eq:even-F} gives
$T-1/k-1/g\ge(t+1)/(2(d-2))$. Since
$(t-1)/(2d)+1/(d-2)\le(t+1)/(2(d-2))$, we have $R<T$.
There is no restriction on the endpoint sizes other than those already
stated.
\end{proof}

\subsection{Intervals ending at a single cut of size two}
\begin{lemma}\label{lem:even-single-two}
Suppose that an interval has a small left end, a single right cut
of size two, and length $r\ge d$. Then
\eqref{eq:even-method2} gives a denominator $R<T$.
\end{lemma}
\begin{proof}
In reverse order, a cut of size two followed by a high cut also gives
the comparison $J_2$. Indeed, at most one of its two crossing edges
ends at the first vertex: moving that vertex past the cut gives
$q_1=d+2-2c\ge d$, hence $c\le1$. Therefore row zero has entries
at least $2,1,0,\ldots,0$; row one has diagonal at least $d$ and
entry at column $j\ge2$ at least $d-j$, by counting the first
vertex's remaining neighbors. The other rows follow from
\eqref{eq:even-overlap}.

Set the first exterior difference to zero and reverse again. On the
last $d$ differences the resulting entries are
\begin{align}
 U_{ij}&=(d-j)(i+1)-2\quad(0\le i\le j\le d-2),\label{eq:even-U}\\
 U_{i,d-1}&=i\quad(0\le i\le d-2),\notag\\
 U_{d-1,d-1}&=d.\notag
\end{align}
For $r=d$, \eqref{eq:even-cut} forces $k=d-2$. Weights at indices
$0,d-2,d-1$ equal to
$1/(d-2)$, $2/((d-2)(d+2))$, and $1/(d+2)$
give $\sum_jU_{ij}w_j=1$ in every row. Their sum
$2D/((d-2)(d+2))$ is smaller than $T=2/(d-2)$.

For $r=d+t$, $t\ge1$, use the initial weights of
\Cref{lem:even-initial-weights} with total \eqref{eq:even-prefix-trimmed}. Put the three tail weights at indices
$t+1,t+d-2,t+d-1$ equal to
\begin{align}
 A&=\frac1{2(d-2)},\label{eq:even-singleweights}\\
 B&=\frac5{2(d-2)(d+2)},\notag\\
 C&=\frac{d-3}{(d-2)(d+2)}.\notag
\end{align}
For tail rows $1\le s\le d-2$, \eqref{eq:even-U} gives contribution
$2(d-1)A+(2B+C-2A)s=(d-1)/(d-2)>1$; the last row gives
$A+(d-2)B+dC=1$. Row $t$ is high, and
$K_{t,t+1}\ge d-2$ by \eqref{eq:even-two}, so $A$ supplies the
missing half there. The initial rows are covered by \Cref{lem:even-initial-weights}.
The total is
\[
 R=\frac1k+\frac{t-1}{2d}
                 +\frac{3d+1}{2(d-2)(d+2)}.
\]
Using \eqref{eq:even-F} once more,
\[
 T-R\ge\frac{t-1}{d(d-2)}+\frac5{2(d-2)(d+2)}>0.
\]
This proves all the single-size-two comparisons by explicit sums of
the Fiedler equations.
\end{proof}

\subsection{The remaining short intervals}\label{sec:even-remaining}
The remaining comparisons can be proved by four weighted equations,
with alternating weights added when the interval is longer. The sums of the relevant overlap bounds are constant, which is why
these few weights cover every row.

\begin{lemma}\label{lem:even-uniform-short}
Let $d\ge6$ be even and $d\le r\le d+5$. If the right end is paired
with $g+l\ge d+2$, or is a single cut of size $l\ge4$, then
\eqref{eq:even-method2} gives a denominator $R<T$.
\end{lemma}
\begin{proof}
Use relative cut indices $0,\ldots,r$, with $q_0=k$ and $q_r=l$.
The interval consists of indices $0,\ldots,r-1$. At a paired right end,
$q_{r-1}=g$; all other cuts strictly between the two ends are high.
As before, $D=d+1$ and $\beta=2(d-2)$.

\paragraph{The initial weights.}
Write $r=d+t$ in the single case and $r=D+t$ in the paired case.
For $t\ge1$, use the initial weights in
\Cref{lem:even-initial-weights} without the optional decrease.
Their total satisfies \eqref{eq:even-initial-four}, and the tail
only has to supply the other half of row $t$. For $t=0$, initially
use only $w_0=1/k$; the paired case will decrease this weight at
the end of its proof.

\paragraph{A single right cut.}
Write $r=d+t$, put $b=d-l\ge2$, and let $H=2(b+1)$.
In addition to the initial weights, put
\begin{equation}\label{eq:even-single-four-simple}
 w_{t+1}=w_{t+d-2}=\frac1H,\qquad
 w_{t+d-1}=\frac1{2d}.
\end{equation}
For $1\le i\le d-2$, the reversed overlap inequality, applied at
cut $r$, gives
\[
 K_{t+1,t+i}\ge2(d-i)-l,\qquad
 K_{t+i,t+d-2}\ge2(i+1)-l.
\]
Their sum is $H$, so these rows are covered by the two weights $1/H$.
Individual lower bounds may be negative, but their sum remains a
valid lower bound. Also, the cut inequality gives
\[
 q_{t+1},q_{t+d-2}\ge2d-2-l\ge H,
\]
where the last inequality is exactly $l\ge4$.
Consequently the last row receives at least $1/2$ from the weight
at $t+d-2$ and at least $1/2$ from its own diagonal. When $t\ge1$,
$K_{t,t+1}\ge q_{t+1}/2\ge H/2$, so the first tail weight supplies
the missing half in row $t$. Thus all rows are covered.

The sum of the weights satisfies
\[
 R\le\frac1k+\frac{t}{2d}+\frac1{b+1}+\frac1{2d}.
\]
Since $1/k-\phi(k)=F(k)\le d/[4(d-2)]$ and $l\ge4$, we obtain
\begin{align*}
 T-R
 &\ge \frac14+\frac1{2b}-\frac1{2l}-\frac1{b+1}
       +\frac{t+1}{d(d-2)}\\
 &\ge \frac18+\frac1{2b}-\frac1{b+1}
       +\frac{t+1}{d(d-2)}\\
 &=\frac{(b-1)(b-2)+2}{8b(b+1)}
       +\frac{t+1}{d(d-2)}>0.
\end{align*}
This proves every single-cut comparison under consideration.

\paragraph{A paired right end.}
Here $r=D+t$, and $g+l\ge d+2$ implies $g,l\ge4$.
Put $b=d-l\ge2$ and $H=2d-g-l=d+b-g$. Thus
\[
 b+2\le g\le d-2,\qquad 4\le H\le d-2.
\]
In addition to the initial weights, put
\begin{equation}\label{eq:even-pair-four-simple}
 w_{t+1}=w_{t+d-1}=\frac1H,\qquad
 w_{t+d}=\frac{d-g}{gH}.
\end{equation}
For $1\le i\le d-1$, apply the reversed overlap inequality at cuts
$r-1$ and $r$, respectively, to obtain
\[
 K_{t+1,t+i}\ge2(d-i)-g,\qquad
 K_{t+i,t+d-1}\ge2i-l.
\]
Their sum is $H$, so these rows are covered. The last row satisfies
\[
 K_{t+d-1,t+d}\ge\frac{d+g-l-2}{2}.
\]
Its contribution from the last two weights is therefore at least
\[
 \frac{d+g-l-2}{2H}+\frac{d-g}{H}
 =1+\frac{g+l-d-2}{2H}\ge1.
\]
When $t\ge1$, both cuts $t,t+1$ are high, so
$K_{t,t+1}/H\ge d/(2H)\ge1/2$ supplies the missing half in row $t$.

We need one elementary bound for the sum of the tail weights:
\begin{equation}\label{eq:even-pair-tail-simple}
 V:=\frac2H+\frac{d-g}{gH}=\frac{d+g}{gH}
 \le\frac14+\frac1{d-2}-\phi(l).
\end{equation}
Indeed,
\[
 gH-(d-2)(b+2)=(g-b-2)(d-2-g)\ge0,
\]
and $d+g\le2(d-1)$. Hence
$V\le U:=2(d-1)/[(d-2)(b+2)]$.
Since $d=b+l$ and $b\ge2,l\ge4$, we also have
\[
 l(b+2)-4(d-2)=(b-2)(l-4)\ge0.
\]
It follows that
\begin{align*}
 \frac14+\frac1{d-2}-\phi(l)-U
 &=\frac{(b-2)^2}{4b(b+2)}
   +\frac{b}{(d-2)(b+2)}-\frac1{2l}\\
 &\ge\frac{(b-2)^2}{4b(b+2)}
      -\frac{(b-2)^2}{2l(b+2)^2}\\
 &\ge\frac{(b-2)^2(b+4)}{8b(b+2)^2}\ge0.
\end{align*}
This proves \eqref{eq:even-pair-tail-simple} using only products
of nonnegative factors.

If $t\ge1$, \eqref{eq:even-initial-four} and
\eqref{eq:even-pair-tail-simple} give
\[
 R\le\frac1k+\frac{t}{2d}+V,
 \qquad
 T-R\ge\frac{t}{d(d-2)}>0.
\]
If $t=0$, decrease the initial weight to
\[
 w_0=\frac1k-\frac1{2d}>0.
\]
The first tail weight is now at index $1$. Row $0$ remains covered,
since
\[
 k w_0+K_{01}w_1
 \ge1-\frac{k}{2d}+\frac{k}{2H}\ge1.
\]
All other rows were covered without this initial weight, so their
inequalities are unchanged. The total is now
$R=1/k+V-1/(2d)$, and consequently
\[
 T-R\ge\frac1{2d}>0.
\]
This completes every paired comparison and the proof.
\end{proof}

\subsection{Long intervals and free ends}\label{sec:even-long}
In this subsection, we bound intervals with a long run of high cuts
and intervals that meet an end of the ordering.

Put
$m=d/2+3$, $c=d(d-2)/(d+2)$, and
$\sigma=(d-4)/(2d(d-2))>0$.
\begin{lemma}\label{lem:even-half}
If a cut of size $k<d$ is followed by $m-1$ high cuts, then on these
$m$ differences
\begin{align}
 Q_I&\ge\frac{\Delta_I^2}{1/k+1/c},\label{eq:even-half}\\
 \frac1k+\frac1c&\le\frac m\beta+\phi(k)-\sigma.\notag
\end{align}
For a nonzero vector, equality in both comparisons requires
$k=d-2$, and the first and last differences are positive.
\end{lemma}
\begin{proof}
Use relative indices $0,\ldots,m-1$, and put
$Z=a_2+\cdots+a_{m-1}$. Keep the terms
\[
 ka_0^2+da_1^2+ka_0a_1+da_1a_2+C(k),
\]
where
\[
 C(k)=\sum_{i,j=2}^{m-1}
 \{\min(i,j)(D-\max(i,j))-k\}a_i a_j.
\]
They are valid lower bounds by \eqref{eq:even-overlap} and
\eqref{eq:even-two}. Negative coefficients, if present in a lower bound,
do not cause a problem, since the differences are nonnegative.
The complete-graph identity
$\sum_{i<j}(v_i-v_j)^2=D\sum_i(v_i-\bar v)^2$, applied with the
first two values zero and the last $(d-2)/2$ values $Z$, gives
\[
 C(d-2)\ge cZ^2+\frac{d+2}{2}
       \left(a_2-\frac{d-2}{d+2}Z\right)^2.
\]
For completeness, the free values have minimizing value
$(d-2)Z/(d+2)$. Subtract that value from them. The remaining sum is
$D\sum v_j^2-(\sum v_j)^2\ge(d+2)\sum v_j^2/2$;
keep its first square.
Put $X=a_0+a_1/2$, $W=Z+a_1/2$. Completing squares now yields
\begin{align*}
 Q_I-kX^2-cW^2\ge{}&
 \frac{d+2}{2}\left(a_2-\frac{d-2}{d+2}Z+
                         \frac{d}{d+2}a_1\right)^2
 +\frac{5d+2}{2(d+2)}a_1^2\\
 &+(d-2-k)(Z^2+a_1^2/4).
\end{align*}
Since $X+W=\Delta_I$, Cauchy--Schwarz proves the first assertion.
The second follows from \eqref{eq:even-F}, since
$m/\beta-F(2)-1/c=\sigma$.
If equality holds, then $a_1=0$, $k=d-2$, and equality in the
complete-graph calculation gives
\begin{align*}
 a_0&=\frac{d}{2D}\Delta_I,\\
 a_2&=\frac{d-2}{2D}\Delta_I,\\
 a_{m-1}&=\frac2D\Delta_I,
\end{align*}
with all other differences zero. This proves the required endpoint
assertion as well.
\end{proof}

\begin{lemma}\label{lem:even-high-runs}
For $u\ge1$ consecutive high cuts,
\begin{equation}\label{eq:even-high}
 Q_I\ge d\left(\sum a_i^2+\sum a_i a_{i+1}\right)
       \ge\frac{2d}{u+1}\Delta_I^2.
\end{equation}
Moreover,
\begin{equation}\label{eq:even-highslack}
 \frac{u+1}{2d}-\frac u\beta
 =\frac{d-2-2u}{2d(d-2)}\le\sigma,
\end{equation}
strictly if $u>1$.
\end{lemma}
\begin{proof}
The diagonal and adjacent terms give the first bound in
\eqref{eq:even-high}. Twice the sum in parentheses is
$a_1^2+a_u^2+\sum(a_i+a_{i+1})^2$; the $u+1$ quantities being
squared sum to $2\Delta_I$. Cauchy--Schwarz proves the second bound.
Subtraction gives \eqref{eq:even-highslack}, whose last inequality
is strict exactly when $u>1$.
\end{proof}

\begin{lemma}\label{lem:even-single-end}
A part of $m$ consecutive high cuts ending immediately before a
single small cut of size $l$ has denominator at most
$m/\beta-\phi(l)-\sigma$.
\end{lemma}
\begin{proof}
To see this for $d\ge8$, reverse the
ordering, keep the first high diagonal term $d a_1^2$, and minimize
the complete-graph comparison on the remaining differences. A
sufficient denominator is
\[
 \frac1d+\frac1{2d-6-8/d-l}\le\frac1d+\frac1{d-l}.
\]
Indeed, after the first high difference, write $Z$ for the sum of the
remaining differences. The complete-graph comparison has two fixed
zero values and $(d-4)/2$ fixed values $Z$. Minimizing over its other
values gives $(2d-6-8/d-l)Z^2$. Adding the retained first diagonal
term and applying weighted Cauchy--Schwarz gives the denominator above.
Its saving from $m/\beta-\phi(l)$ is at least
$\sigma+4/[d(d-2)]$, by \eqref{eq:even-F}. For $d=6$ use
\eqref{eq:even-high}, with denominator $7/12$; this is at most the
required denominator for either $l=2$ or $4$.
\end{proof}

\begin{lemma}\label{lem:even-long-free}
The local methods give a denominator $R<T$ for an interval with
two small ends and $r\ge2m=d+6$, for an interval with exactly one
free end and $r\ge d$, and for an interval with two free ends and
$r\ge d/2$. Here $T$ is defined by \eqref{eq:even-T}, with the
free-end terms omitted.
\end{lemma}
\begin{proof}
For an interval with two small ends and $r\ge2m=d+6$, take $m$
differences at each end and the high cuts between them. The boundary parts save $2\sigma$, whereas the middle loses
at most $\sigma$ by \eqref{eq:even-highslack}. At a paired right end,
use $\phi(g)\le-\phi(l)$, since $g+l\ge d$.

If exactly one end is free and $r\ge d$, combine its high part with
the part of length $m$ at the small end. Only a high part of length
one could remove the entire saving. Equality then forces positive
differences on both sides of the join, by \Cref{lem:even-half};
the omitted term $2K_{i,i+1}a_i a_{i+1}\ge d a_i a_{i+1}$ is
strictly positive. At a single small right end there is already an
extra saving for $d\ge8$. The remaining possibility $d=6,r=7$
consists of seven high cuts: its denominator $2/3$ is smaller than
$7/8-\phi(l)\ge3/4$. The case with no high part saves $\sigma$.
Finally, when both ends are free and $r\ge d/2$, \eqref{eq:even-high}
gives coefficient at least $2d^2/(d+2)>\beta$. This proves all the
long-interval and free-end comparisons, with strictness.
\end{proof}

\subsection{The common partition and its sharp intervals}
All the local comparisons are now available. The only intervals not
proved strict have length $D$ and end in $P$ or $Q$. If they also
start in one of these types, both types agree. We call such a
$P\longrightarrow P$ or $Q\longrightarrow Q$ interval a
\emph{standard gap}.

\begin{lemma}\label{lem:even-standard-block}
Every standard gap is an actual block isomorphic to $M_d$, with
its two separating vertices as singleton endpoints.
\end{lemma}
\begin{proof}
Let its endpoints be $a,b=a+D$, each with the same number
$k\in\{2,d-2\}$ of neighbors to its right. The set
$T=\{a+1,\ldots,b-1\}$ has $d$ vertices, and every edge leaving
$T$ meets $a$ or $b$. Let $e_0$ be one if $a,b$ are adjacent and
zero otherwise. Counting the degrees in $T$ gives
\[
 d^2=2e(T)+k+(d-k)-2e_0.
\]
Since $e(T)\le\binom d2$, necessarily $e_0=0$ and $T$ is a
clique. Every vertex of $T$ then has exactly one neighbor among
$a,b$, with $k$ adjacent to $a$ and $d-k$ to $b$.
The induced graph is therefore
$K_1+K_k+K_{d-k}+K_1$, which is $M_d$ up to reversal.
It has no cut vertex because $k,d-k\ge2$. No internal vertex has
an outside neighbor, and neither endpoint can be bypassed in the
ordering. It is consequently a maximal subgraph without a cut
vertex, hence an actual block.
\end{proof}

The following lemma packages the local bounds. It is stated for
arbitrary nondecreasing vectors: only the later global comparison
uses the Fiedler equations.

\begin{lemma}\label{lem:even-certificate}
Consider consecutive difference indices in an increasing ordering
of a $d$-regular graph. Each of the two boundary positions is either
an end of the ordering or a separating vertex of type $P$ or $Q$;
in particular, the entire ordering is allowed.
For a one-ended subgraph it is enough that all vertices away from its
separating boundary have degree $d$.
There is a partition into intervals of differences of lengths at most
$a_d$, with positive numbers $R_I$, such that
\begin{align}
 \Delta_I&\le R_I\max_{i\in I} c_i,\label{eq:even-certificate}\\
 R_I&\le T_I=\frac{|I|}{\beta}+\psi_I^- -\psi_I^+.\notag
\end{align}
Here $c_i=\sum_jK_{ij}a_j$, $\psi=\phi(k)$ at a recorded small cut
of size $k$, and $\psi=0$ at a free partition point. The values of
$\psi$ agree on the two sides of a common partition point.
The numbers $R_I$ are bounded above by a constant depending only on $d$.
There is a constant $\gamma_d>0$ such that each interval is either
strict, with $R_I\le T_I-\gamma_d$, or has length $D$ and ends in
$P$ or $Q$. In the latter case, if it starts in $P$ or $Q$, its two
types agree and it is a standard gap.

If an initial segment of length $L$ contains no standard gap, then
at least $c_dL-C_d$ of the partition intervals contained in this
segment are strict. The same assertion holds for the whole interval
between two separating vertices when it contains no standard gap.
\end{lemma}
\begin{proof}
Put $h=d/2$. Between two successive last cuts of small groups,
keep the whole interval if its length is less than $2d+h$.
Otherwise keep the first and last $d$ differences as separate
intervals and partition the middle high run into lengths from
$h$ to $2h-1$. This is possible: write its length as $qh+t$,
where $0\le t<h$, and add $t$ to one of the $q$ parts.

Before the first small group and after the last, keep the whole
interval if its length is less than $d+h$. Otherwise keep $d$
differences at the small end and partition the remaining high run
in the same way. The first and last $d$ cuts of a full ordering
are high, since $q_i\ge i(D-i)\ge d$ when one side has
$i\le d$ vertices. The same argument applies at a genuine end of
a subgraph whose nonboundary vertices have degree $d$.
If there is no small cut, partition the whole high run; its length
is at least $h$ in all applications here.

These rules cover every difference index exactly once. Each interval
has length at most $a_d=5d/2$. Intervals with two small ends have
length at least $d$, and at least $D$ if the right end is paired.
Those with one small end have length at least $d$, and high-only
intervals have length at least $h$. The preceding local lemmas
therefore cover every possible interval.

The inequalities \eqref{eq:even-method1} and
\eqref{eq:even-method2} prove the first assertion for every increasing
vector. Neither method uses the eigen-equation. All the estimates
and strictness checks needed for $R_I\le T_I$ were proved in
\Cref{lem:even-shortest-separator,lem:even-longer-separator,lem:even-single-two,lem:even-uniform-short,lem:even-long-free}. Only finitely
many sizes, endpoint types and lengths occur for fixed $d$. Compactness on
$\{a_i\ge0:\sum a_i=1\}$ in the strict quadratic-form cases gives
one positive $\gamma_d$ for all strict intervals.

Every interval not proved strict ends in $P$ or $Q$, and two
successive intervals of this kind make the second a standard gap.
The first such interval is also standard if it starts at the given
separating vertex. Therefore, in a segment without standard gaps,
at most one more interval is nonstrict than strict. The intervals
contained in an initial segment cover all but at most $a_d$
differences of that segment. Their number is at least
$L/a_d-1$, which proves the last assertion after changing constants.
\end{proof}

\subsection{Comparison graphs and an upper bound}\label{sec:upper}
We now give explicit end completions for the comparison
chains of \cite[Theorem~1.8 and Table~1]{AbGhDiam}. This gives a
regular graph at every sufficiently large order and the $O_d(n^{-3})$
error term needed for the exact-minimizer arguments.

For this construction, allow any even $d\ge4$, and put $D=d+1$ and $\beta=2(d-2)$.

For $D\le N\le2D-1$ and $p\in\{2,d-2\}$, start with the graph on
$\{0,\ldots,N-1\}$ in which each vertex is joined to its $d/2$ nearest
neighbors on either side in the cyclic order. Equivalently, the allowed
differences modulo $N$ are $\pm1,\ldots,\pm d/2$.
Delete the disjoint edges $01,23,\ldots,(p-2)(p-1)$ and join all
$p$ of their endpoints to a new vertex $r$. Call this graph $B_{N,p}$,
and call $r$ its root.

\begin{lemma}\label{lem:comparison-completions}
For $D\le N\le2D-1$ and $p\in\{2,d-2\}$, the graph $B_{N,p}$
is simple and connected. Its root has degree $p$, and every other
vertex has degree $d$. Deleting any one remaining edge of difference
one leaves it connected.
\end{lemma}
\begin{proof}
Each endpoint of a deleted edge loses one neighbor and gains the
root as a neighbor, so every nonroot degree remains $d$. The root
has exactly the $p$ specified neighbors. The construction is simple.
To verify connectedness, observe that the edges of difference two form one
spanning cycle for odd $N$ and two cycles for even $N$. None of those
edges is deleted. In the even case, some edges of difference one still
join the two cycles. This remains true after deleting any one further
edge of difference one, since at least $N-p/2-1>0$ such edges remain.
The root retains its incident edges.
\end{proof}

\begin{lemma}\label{lem:comparison-orders}
For every $n\ge5D$, a chain of $k\ge2$ consistently oriented $M_d$
blocks can be completed at its ends to a connected simple $d$-regular
graph $G_{n,d}$ of order $n$, where $n=Dk+O_d(1)$.
\end{lemma}
\begin{proof}
Take a chain of $k$ consistently oriented $M_d$ blocks. Identify its
left singleton, of internal degree two, with the root of $B_{N,d-2}$,
and its right singleton with the root of $B_{D,2}$. The resulting graph
$G_{n,d}$ is simple, connected and $d$-regular, and has
\[
 n=kD+N+D+1.
\]
As $N$ runs through $D,\ldots,2D-1$, the right side covers every residue
modulo $D$. Thus the construction exists for every $n\ge5D$, with
$k\ge2$ and $n=Dk+O_d(1)$.
\end{proof}

We record the error term needed to apply stability to exact
minimizers. The calculation uses only the chain and bounded end
subgraphs; it does not assume a structure theorem for a minimizer.
\begin{lemma}\label{lem:comparison-upperbound}
There is a constant $A_d>0$ such that every comparison graph $G_{n,d}$
from \Cref{lem:comparison-orders} satisfies
\begin{equation}\label{eq:comparisonerror}
 \mu(G_{n,d})\le\frac{\beta\pi^2}{n^2}+\frac{A_d}{n^3}.
\end{equation}
\end{lemma}
\begin{proof}
Choose successive endpoint values
$x_j=-\cos(\pi j/k)$ for $0\le j\le k$, extend them within each group
using \eqref{eq:comparisonperiod}, and extend the values constantly over
the end pieces. Denote the resulting vector by $\y$. Its quadratic form
is
\begin{align*}
 \y^\top L(G)\y
 &=\frac{\beta}{D}\sum_{j=0}^{k-1}(x_{j+1}-x_j)^2\\
 &=\frac{2\beta k}{D}\sin^2\frac\pi{2k}
 =\frac{\beta\pi^2}{2Dk}+O_d(k^{-3}).
\end{align*}
Assign $D$ vertices to each repeated group, counting a common endpoint
only in the next group. Replacing all their values by the first endpoint
value changes the sum of their components and the sum of their squares
by at most $C_d|x_{j+1}-x_j|$. Since $|x_j|\le1$ and
$\sum_j|x_{j+1}-x_j|=2$, the total error is $O_d(1)$.
The elementary cosine sums therefore give
\begin{align*}
 \sum_i y_i&=O_d(1),\\
 \sum_i y_i^2&=\frac n2+O_d(1),\\
 \|\y-\bar y\1\|^2&=\frac n2+O_d(1).
\end{align*}
The $O_d(1)$ end pieces do not change any of these estimates.
Equations \eqref{eq:Rayleigh} and \eqref{eq:centering} now give
\eqref{eq:comparisonerror} after increasing $A_d$ if necessary.
\end{proof}
Thus an exact minimizer satisfies the upper hypothesis needed for
\Cref{thm:evenmain}, with $\e=O_d(n^{-1})$. This is the only
numerical input required for the exact structural theorem.

\section{Structure of graphs with small algebraic connectivity}
\label{sec:evenstructure}\label{sec:evensharp}\label{sec:even-global}
The comparison upper bound is enough to begin the structural proof.
We first treat all graphs close to that bound, without exact minimality.

\begin{theorem}\label{thm:evenmain}
Fix an even $d\ge6$ and put $\beta=2(d-2)$. There are constants
$C_d,n_d$ such that, if $G$ is $d$-regular of order $n\ge n_d$ and
\[
 n^2\mu(G)\le\beta\pi^2+\e,\qquad 0\le\e\le1,
\]
then actual blocks isomorphic to $M_d$, on one path of the
block-tree, cover at least
\begin{equation}\label{eq:even-cover}
 n-C_d(\e+n^{-1})^{1/3}n
\end{equation}
vertices. Every exact minimizer therefore has
$n-O_d(n^{2/3})$ vertices in these blocks and satisfies
$\dm(G)\ge3n/(d+1)-O_d(n^{2/3})$.
\end{theorem}

\subsection{Comparison with a path at shifted positions}
The boundary correction in \eqref{eq:even-T} has a direct geometric
meaning: moving the two endpoints turns $\beta T_I$ into the
length of the interpolation interval. This avoids estimating the
boundary terms one at a time.

\begin{lemma}\label{lem:even-shifted-path}
Fix $K>0$. Let $\x$ be an increasing mean-zero unit Fiedler vector
of a $d$-regular graph with $\mu(G)\le K/n^2$. Use the partition
from \Cref{lem:even-certificate} on the whole ordering.
There is an increasing piecewise-linear function $Y$ on $[1,n]$,
with $y_j=Y(j)$, such that
\begin{align}
 \|\x-\y\|&=O_{d,K}(n^{-1}),\notag\\
 \|\y-\bar y\1\|^2&=1+O_{d,K}(n^{-1}),\notag\\
 \mu(G)&\ge\beta\int_1^nY'(t)^2\,dt
   +\kappa_d\int_{\mathcal B}Y'(t)^2\,dt-C_{d,K}n^{-3}.
 \label{eq:even-shifted-statement}
\end{align}
Here $\kappa_d>0$, and $\mathcal B$ is the union of the shifted
exceptional intervals. Each shifted interval has length at least
three and at most a constant depending on $d$; its old and new
lengths are comparable by constants depending only on $d$.
\end{lemma}
\begin{proof}
By \Cref{lem:smallentries}, $\max_i|x_i|=O_{d,K}(n^{-1/2})$.
The prefix sum in \eqref{eq:prefixsum} has magnitude at most
$\sqrt n/2$. Since $c_i=\mu f_i$, $q_i\ge2$, and $c_i\ge q_i a_i$,
we have
\begin{align}
 0\le a_i&\le\frac{\mu(G)\sqrt n}{4}=O_{d,K}(n^{-3/2}),
 \label{eq:even-errors}\\
 |c_j-c_i|&=O_{d,K}(n^{-5/2})\quad\text{if }|i-j|=O_d(1).\notag
\end{align}
In the rest of this proof, $K$ is fixed and absorbed into the constants. For fixed $d$, only finitely many positive values of $T_I$ occur.
By decreasing $\gamma_d$ and increasing the chosen $R_I$ when
necessary, we may therefore assume
$c_d\le R_I\le T_I\le C_d$ for some $c_d>0$, while retaining
$R_I\le T_I-\gamma_d$ on every strict interval.

\emph{The contribution of one interval.}
For an interval $I$ beginning at $a$, let
$E_I=\sum_{i\in I}a_i c_i$. By \eqref{eq:even-errors},
\begin{equation}\label{eq:even-localerror}
 E_I=c_a\Delta_I+O_d(n^{-4}).
\end{equation}
Either local method gives
$\Delta_I\le R\max_{i\in I}c_i=R c_a+O_d(n^{-5/2})$.
Multiplying by $\Delta_I=O_d(n^{-3/2})$ and using
\eqref{eq:even-localerror}, we obtain
\begin{equation}\label{eq:even-localfinal}
 E_I\ge\frac{\Delta_I^2}{T}-O_d(n^{-4}).
\end{equation}
On a strict interval the bound improves by
$\gamma'_d\Delta_I^2$, with $\gamma'_d>0$ independent of $I$.
Indeed, \Cref{lem:even-certificate} gives $R_I\le T_I-\gamma_d$,
and both denominators are positive and bounded above. Thus
$1/R_I-1/T_I\ge\gamma'_d$. We call these intervals exceptional.
Every other interval has length $D$ and the endpoint-type restriction
in that lemma.

\emph{The corrected positions.}
At each original endpoint $a$ of the partition, prescribe the value
$x_a$ at the new position $a-\beta\phi(k)$ if the endpoint is a
small cut of size $k$. At a free endpoint retain position $a$.
The new length corresponding to an old interval is
\[
 \ell_I=r+\beta\phi(k)-\beta\phi(l)=\beta T.
\]
Since $|\beta\phi(k)|\le(d-4)/2$, the new positions remain in
order. Intervals with two small ends have $\ell_I\ge4$; those with
one small end also have $\ell_I\ge4$; high-only intervals have
$\ell_I=r\ge d/2\ge3$. All new lengths are bounded in terms of $d$.
The first and last positions are $1,n$.

Interpolate linearly between these prescribed values, obtaining an
increasing continuous function $Y$ on $[1,n]$. Put $y_j=Y(j)$ and
$b_i=y_{i+1}-y_i$, as in the companion paper.
Each $Y(j)$ is between original components at indices a bounded
distance from $j$. It follows from \eqref{eq:even-errors} that
\begin{align}
 \|\x-\y\|&=O_d(n^{-1}),\label{eq:even-norm}\\
 \|\y-\bar y\1\|^2&=1+O_d(n^{-1}).\notag
\end{align}
Let $\mathcal B$ be the union of the moved exceptional intervals.
Since $\int_I Y'^2=\Delta_I^2/\ell_I$, summing
\eqref{eq:even-localfinal} gives
\begin{equation}\label{eq:even-globalenergy}
 \mu(G)\ge\beta\int_1^nY'(t)^2\,dt
             +\kappa_d\int_{\mathcal B}Y'(t)^2\,dt-C_dn^{-3},
\end{equation}
where $\kappa_d>0$. Here $\sum_I E_I=\mu(G)$ by
\eqref{eq:even-energy} and \eqref{eq:even-b}; cross terms between
intervals have not been counted independently or discarded with
an uncontrolled sign. There are $O(n)$ intervals, which accounts
for the summed error $O_d(n^{-3})$.
On each unit interval Cauchy--Schwarz gives
\[
 (y_{j+1}-y_j)^2\le\int_j^{j+1}Y'(t)^2\,dt.
\]
Consequently,
\begin{equation}\label{eq:even-discrete-energy}
 \mu(G)\ge\beta\,\y^\top L(P_n)\y-C_dn^{-3}.
\end{equation}
Together with the norm estimate, this completes the comparison.
\end{proof}

\subsection{Near equality and the actual blocks}
\begin{proof}[Proof of \Cref{thm:evenmain}]
Put $\theta=\e+n^{-1}$ and apply \Cref{lem:even-shifted-path} with
$K=\beta\pi^2+1$. The path inequality and the norm estimate give
\begin{equation}\label{eq:even-excess}
 \beta\{\y^\top L(P_n)\y-p_n\|\y-\bar y\1\|^2\}
 +\kappa_d\int_{\mathcal B}Y'^2
 \le C_d\theta n^{-2}.
\end{equation}
Both terms are nonnegative. Thus the norm and energy hypotheses of
\Cref{lem:pathstable} hold. In particular, writing
$b_i=y_{i+1}-y_i$ and $h_i=v_{i+1}-v_i$, we have
\begin{equation}\label{eq:even-pathclose}
 \sum_i(b_i-h_i)^2\le C_d\theta/n^2.
\end{equation}
If $\theta$ is bounded below by a fixed positive constant, the
conclusion is trivial after increasing $C_d$. We may therefore
use the small-$\theta$ assertion of the path lemma.

Let $A$ consist of the integer indices $j$ for which $[j,j+1]$
is contained in one moved exceptional interval. Cauchy--Schwarz on
these disjoint unit intervals and \eqref{eq:even-excess} give
\[
 \sum_{j\in A}b_j^2\le\int_{\mathcal B}Y'^2
                       \le C_d\theta/n^2.
\]
The last assertion of \Cref{lem:pathstable} now gives
$|A|\le C_d\theta^{1/3}n$ directly. A moved interval of length
$\ell\ge3$ contains at least $\ell-2\ge\ell/3$ complete unit
intervals. The interiors of the unit intervals counted for different
partition intervals are disjoint. Hence
\[
 |\mathcal B|\le3|A|=O_d(\theta^{1/3}n).
\]
Old and moved lengths are comparable, so the original exceptional
intervals also have total length $O_d(\theta^{1/3}n)$.

Every remaining interval has length $D$ and right type $P$ or $Q$.
If its left type is neither, the preceding interval is exceptional,
unless it is the first interval. Discard these further intervals.
At most one follows each exceptional interval; their bounded lengths
and the first interval add only $O_d(\theta^{1/3}n)$ differences.
Every interval now left is standard. By
\Cref{lem:even-standard-block}, it is an actual $M_d$ block, and
the separating endpoints place all selected blocks on one path of
the block-tree.

For completeness, passing from differences to vertices loses no
additional order of error. Each standard interval contains $D$
differences and its $D+1$ consecutive vertices. The union of these
vertex sets has at least as many vertices as there are selected
differences. Since all but $O_d(\theta^{1/3}n)$ differences were
selected, the covered vertex count is \eqref{eq:even-cover}.

Finally, \Cref{lem:comparison-upperbound} permits
$\e=A_d/n$ for an exact minimizer, once $n$ is sufficiently large.
The exceptional-vertex bound is then $O_d(n^{2/3})$.
Each selected block must be traversed between its singleton endpoints,
a distance of three. There are at least
$n/D-O_d(n^{2/3})$ selected blocks, proving the stated diameter
bound. No numerical lower bound for the minimum has been used in
this deduction.
\end{proof}

\section{The structure of \texorpdfstring{$\mu$}{mu}-minimal graphs}\label{sec:even-exact}
In this section, we prove \Cref{thm:evenfull} for $d\ge6$. We first
bound the exceptional intervals, then place them at the ends, and
finally exclude branching. Throughout, $d$ is fixed and even,
$D=d+1$, and $\beta=2(d-2)$. We retain the function
$\phi$ and the types $P,Q$ from \Cref{sec:even-cuts}.
For $k\in\{2,d-2\}$, write
\[
 C_k=K_1+K_k+K_{d-k}+K_1.
\]
This is an oriented copy of $M_d$. Its left and right singleton
vertices have degrees $k$ and $d-k$, respectively. When subgraphs
are joined at a singleton, these two degrees must add to $d$.
We will keep track of this condition in every replacement.

Recall that a standard gap is a length-$D$ interval joining
separating vertices of matching type $P$ or $Q$. By
\Cref{lem:even-standard-block}, it is an actual oriented block $C_k$.
The arbitrary-vector interval comparison, including a fixed loss on
strict intervals, was proved once in \Cref{lem:even-certificate}.
We now use it with prescribed endpoint values and degrees.

\subsection{Effective resistance with prescribed boundary degrees}
Call a connected simple graph $H$ with distinct specified vertices
$s,t$ \emph{admissible of type $(k,l)$} if
\begin{align}
 \deg_H(s)&=k,\label{eq:even-terminaldegrees}\\
 \deg_H(t)&=d-l,\notag\\
 \deg_H(v)&=d\quad(v\ne s,t),\notag
\end{align}
where $k,l\in\{2,d-2\}$.
Write $R_H$ for the reciprocal of the minimum of
$\mathcal E_H(\z)=\sum_{uv\in E(H)}(z_u-z_v)^2$ under $z_s=0,z_t=1$.
The minimizing vector is unique and satisfies the zero eigen-equation
at every other vertex. Its values lie in $[0,1]$, by taking a maximum
or minimum in those equations. Multiplying it by $R_H$ gives a
vector whose endpoint difference and quadratic form both equal $R_H$.
For every cut between consecutive values, the sum of the component
differences over the edges crossing that cut is one. This follows by
summing the zero eigen-equations and the equation at the first terminal. We refer to this as the unit-current vector, using only these
stated identities.

\begin{theorem}\label{thm:even-terminal}
If $H$ is admissible of type $(k,l)$ and has $m$ vertices, then
\begin{equation}\label{eq:even-terminal}
 R_H\le\frac{m-1}{\beta}+\phi(k)-\phi(l).
\end{equation}
Equality holds precisely when $k=l$ and $H$ is a chain of copies of
$C_k$, joined at their singleton vertices, with $s,t$ the outer
singletons.
\end{theorem}
\begin{proof}
The internal degree condition implies $m\ge d+2$. Indeed, an internal
vertex needs $d$ neighbors, and if $m=d+1$, every internal vertex
would be adjacent to both terminals, forcing terminal degrees at
least $d-1$.

For the cut counts one may complete $H$ to a regular graph as follows.
At $s$ attach a disjoint $K_D$ with a matching of size $(d-k)/2$
deleted, joining $s$ to all endpoints of that matching. At $t$ do
the same with a matching of size $l/2$. The resulting graph is simple,
connected and $d$-regular. Put the new left vertices before $s$ and
the new right vertices after $t$. At $s$ the two adjacent cuts have
sizes $d-k,k$; at $t$ they have sizes $d-l,l$.

Order the unit-current vector on $H$ increasingly, putting $s$ first
and $t$ last among any ties. Extend it by a constant below zero on
the left completion and by a constant above $R_H$ on the right.
Across every cut within $H$, $c_i=1$: sum the zero eigen-equations
on the internal vertices on one side. The completion edges do not
cross these cuts. \Cref{lem:even-certificate} gives
\[
 R_H=\sum_I\Delta_I\le\sum_I R_I
 \le\sum_I T_I=\frac{m-1}{\beta}+\phi(k)-\phi(l).
\]
The boundary terms cancel at every interior partition point.

If equality holds, there is no strict interval. The first interval
starts in $P$ or $Q$, so it is standard. Induction gives the same
conclusion for every interval, with unchanged orientation. These
intervals cover all differences from $s$ to $t$. The degree count
for a standard gap identifies the whole graph as the stated chain.
Conversely, one copy of $C_k$ has coefficient
\begin{equation}\label{eq:even-cellresistance}
 r_C=\frac1k+\frac1{k(d-k)}+\frac1{d-k}
     =\frac D\beta.
\end{equation}
These coefficients add at shared singleton vertices. A chain of
$q$ copies has $m-1=qD$, proving equality.
\end{proof}

The boundary terms in \eqref{eq:even-terminal} must be retained.
In particular, it is not enough to replace the odd-degree coefficient
by $\beta$ in a two-terminal inequality. They will also enter the
interchange calculation below.

\subsection{An end subgraph and its replacement}
We compare end subgraphs through their contribution at the root.
The replacement test below uses the number of negative eigenvalues,
so it does not require a separate mean correction for a trial vector.

We use the following matrix identity for fixed boundary values.
\begin{lemma}\label{lem:boundary-square}
Let
\begin{align*}
 M&=\begin{pmatrix}C&-J\\-J^\top&B\end{pmatrix},\\
 S&=C-JB^{-1}J^\top,
\end{align*}
where $M$ is real symmetric and $B$ is positive definite.

Then
\begin{equation}\label{eq:boundary-square}
 \begin{pmatrix}\bu\\\bv\end{pmatrix}^{\!\top}
 M\begin{pmatrix}\bu\\\bv\end{pmatrix}
 =\bu^\top S\bu+
 (\bv-B^{-1}J^\top\bu)^\top B(\bv-B^{-1}J^\top\bu).
\end{equation}
For fixed $\bu$, the minimum is therefore $\bu^\top S\bu$, attained
at $\bv=B^{-1}J^\top\bu$. Moreover, $M$ and $S$ have the same number
of negative eigenvalues, counted with multiplicity.
\end{lemma}
\begin{proof}
Direct expansion gives \eqref{eq:boundary-square}. The minimum follows
because $B$ is positive definite. For the last assertion, the invertible change
$(\bu,\bv)\mapsto(\bu,\bv-B^{-1}J^\top\bu)$ makes the quadratic
form block diagonal with blocks $S,B$. An invertible change
preserves the largest dimension of a subspace on which a quadratic
form is negative definite. By diagonalization, this dimension is
its number of negative eigenvalues, and $B$ has none. This is the
case of Sylvester's law of inertia needed here; see
\cite[Theorem~24.2.2]{Spielman}.
\end{proof}
The matrix $S$ is the \emph{Schur complement}; see
\cite[Section~12.7]{Spielman}.

The following lemma gives the corresponding small-parameter formula.
\begin{lemma}\label{lem:boundary-expansion}
Suppose a quadratic
form with fixed boundary values has minimum $q(0)$ at the vector
$\mathbf f_0$ of free values, and its excess at $\mathbf f_0+\mathbf w$
is $\mathbf w^\top B\mathbf w$. Let $\omega$ be the fixed boundary
contribution to the squared norm. If $B-tI$ is positive definite,
the minimum after subtracting $t$ times the squared norm is
\begin{equation}\label{eq:boundary-expansion}
 q(t)=q(0)-t(\|\mathbf f_0\|^2+\omega)
             -t^2\mathbf f_0^\top(B-tI)^{-1}\mathbf f_0.
\end{equation}
The minimizing change is $\mathbf w=t(B-tI)^{-1}\mathbf f_0$.
\end{lemma}
\begin{proof}
The terms depending on $\mathbf w$ are
$\mathbf w^\top(B-tI)\mathbf w-2t\mathbf f_0^\top\mathbf w$.
Their minimum is the last term in \eqref{eq:boundary-expansion},
by \eqref{eq:boundary-square}, and the equality condition there gives
the stated minimizing change.
\end{proof}

An admissible end subgraph is a connected simple graph with a root $r$ of degree
$p\in\{2,d-2\}$ and $b$ other vertices, all of degree $d$.
The root is shared with the rest of the graph. Put
\begin{align*}
 A_H&=L(H)[V(H)\setminus\{r\}],\\
 T_H(t)&=\1^\top(A_H-tI)^{-1}\1.
\end{align*}
Here $A_H$ is a principal submatrix, so its diagonal entries are $d$.
If $\mathbf a$ indicates the neighbors of $r$, then $A_H\1=\mathbf a$
and $\1^\top A_H\1=p$.

\begin{lemma}\label{lem:even-response}
The smallest eigenvalue of $A_H$ is at least $b^{-2}$.
For $0\le tb^2\le1/2$, the vector
$\bu=(A_H-tI)^{-1}\1$ is positive and satisfies
\begin{align}
 \max_vu_v&\le2b^2,\label{eq:even-responsebound}\\
 T_H(t)&\le2b^3.\notag
\end{align}
Also $T_H(0)\le T_H(t)\le T_H(0)/(1-tb^2)$, $T_H(0)\le b^3$, and
\begin{align}
 \|(A_H-tI)^{-1}\|&\le2b^2,\label{eq:resolventbounds}\\
 \|(A_H-tI)^{-1}\1\|^2&\le4b^5,\notag\\
 0\le T_H(t)-T_H(0)&\le2tb^5.\notag
\end{align}
The matrix norm is its Euclidean operator norm.
For every $0\le t<\lambda_{\min}(A_H)$, if the component at the
root is fixed at $a$, the vector minimizing
$\mathcal E_H- t\sum_{v\ne r}x_v^2$ is
\begin{equation}\label{eq:even-rootvector}
 \x_{H-r}=a\{\1+t(A_H-tI)^{-1}\1\},
\end{equation}
and the minimum is $-a^2\{bt+t^2T_H(t)\}$.
\end{lemma}
\begin{proof}
Extend a vector by zero at $r$. Choose a path of at most $b$ edges
from each vertex to $r$; Cauchy--Schwarz and summation give
$\|\z\|^2\le b^2\mathcal E_H(\z)$. This proves the eigenvalue bound.
For $\lambda\ge b^{-2}$ and $0\le tb^2\le1/2$,
\[
 \frac1\lambda\le\frac1{\lambda-t}
 \le\frac1{1-tb^2}\frac1\lambda\le2b^2.
\]
Diagonalization gives the asserted bounds for $T_H(t)$ and the
first two estimates in \eqref{eq:resolventbounds}, using
$\|\1\|^2=b$. The last follows from
\[
 (A_H-tI)^{-1}-A_H^{-1}
 =tA_H^{-1}(A_H-tI)^{-1}.
\]

For any positive definite matrix $Q$ with nonpositive off-diagonal
entries, the minimizer of
$\frac12\z^\top Q\z-\1^\top\z$ is nonnegative: replacing $\z$ by
$|\z|$ does not increase the quadratic term and improves the linear
term if a component is negative. A zero component of its minimizer
would give $(Q\z)_v\le0$, whereas $Q\z=\1$. Thus
$Q^{-1}\1>\0$. Apply this to $Q=A_H-tI$.

Set $u_r=0$ and $U=\max u_v$. At each nonroot vertex,
$(L(H)\bu)_v=1+tu_v>0$. Summing over the upper side of a cut in the
increasing order bounds every edge difference by $b(1+tU)$.
A path from $r$ to a maximum therefore gives
$U\le b^2(1+tU)$, and hence $U\le2b^2$. Summing proves the other
bound in \eqref{eq:even-responsebound}.

With the root fixed at $a$, the quadratic expression is
$\z^\top(A_H-tI)\z-2a\mathbf a^\top\z+pa^2$.
By \eqref{eq:boundary-square}, its minimizer is
$a(A_H-tI)^{-1}\mathbf a$. Since $A_H\1=\mathbf a$,
this is \eqref{eq:even-rootvector}. Also
\[
 \mathbf a^\top(A_H-tI)^{-1}\mathbf a
 =p+bt+t^2T_H(t),
\]
by writing $A_H=(A_H-tI)+tI$ on both sides of the inverse.
Substitution gives the asserted minimum.
\end{proof}

\begin{lemma}\label{lem:even-replaceend}
Let $G$ have an admissible end subgraph $H$ with $b$ nonroot vertices,
and let $\x$ be a mean-zero unit Fiedler vector with $x_r\ne0$.
Put $\mu=\mu(G)$, and replace $H$ by an admissible end subgraph
$H'$ with the same $b$ and root degree. Suppose
\[
 \mu<\min\{\lambda_{\min}(A_H),\lambda_{\min}(A_{H'})\}.
\]
If $T_{H'}(\mu)>T_H(\mu)$, then the resulting connected regular
graph $G'$ has smaller algebraic connectivity.
\end{lemma}
\begin{proof}
Order the exterior vertices, including the root $r$, first. Put
$B_H=A_H-\mu I$ and let $\mathbf a$ indicate the neighbors of
$r$ in $H-r$. Then
\[
 L(G)-\mu I=
 \begin{pmatrix}C&-e_r\mathbf a^\top\\-\mathbf a e_r^\top&B_H\end{pmatrix}.
\]
The exterior block $C$ is unchanged by replacement: the exterior
edges and the degree of the root are unchanged. Both $B_H$ and
$B_{H'}$ are positive definite by hypothesis.
Writing $p=\deg_H(r)$ and
$\rho_H=\mathbf a^\top B_H^{-1}\mathbf a$,
\Cref{lem:even-response} gives
\begin{align*}
 \rho_H&=p+b\mu+\mu^2T_H(\mu),\\
 \eta:=\rho_{H'}-\rho_H&=\mu^2\bigl(T_{H'}(\mu)-T_H(\mu)\bigr)>0.
\end{align*}
By \eqref{eq:boundary-square}, the reduced exterior matrices are
$S=C-\rho_He_re_r^\top$ and $S'=S-\eta e_re_r^\top$.
Since $L(G)-\mu I$ has exactly one negative eigenvalue, so does
$S$. The exterior restriction $\z$ of $\x$ satisfies
$S\z=\0$ and $z_r\ne0$.

Choose a unit eigenvector $\bu$ of $S$ with eigenvalue
$-\sigma<0$; then $\bu\perp\z$. On their two-dimensional span,
\begin{equation}\label{eq:even-endchange}
 (\xi\bu+\zeta\z)^\top S'(\xi\bu+\zeta\z)
 =-\sigma\xi^2-\eta(\xi u_r+\zeta z_r)^2<0
 \quad\text{if }(\xi,\zeta)\ne(0,0).
\end{equation}
The first term is negative when $\xi\ne0$; otherwise the
second is negative because $z_r\ne0$. Thus $S'$ has at least
two negative eigenvalues. By \eqref{eq:boundary-square}, so does
$L(G')-\mu I$, proving $\mu(G')<\mu(G)$.
\end{proof}

\begin{lemma}\label{lem:even-endloss}
Fix $0<\xi\le1/2$. Order the vector from
\Cref{lem:even-response}, together with $u_r=0$, increasingly.
Suppose its first $\lfloor\xi b\rfloor$ differences contain no
standard gap. If $0\le tb^2\le1/2$, then
\begin{equation}\label{eq:even-endloss}
 T_H(t)\le\frac{b^3}{3\beta}-c_{d,\xi}b^3+C_{d,\xi}b^2+C_dt b^5,
\end{equation}
where $c_{d,\xi}>0$.
\end{lemma}
\begin{proof}
Index the root by $0$ and the other vertices by $1,\ldots,b$.
Put $a_i=u_{i+1}-u_i$, $w_i=b-i$. Summing the equations on the
upper side of cut $i$ gives
\begin{align*}
 c_i&=\sum_{v>i}(1+tu_v),\\
 T_H(t)&=\sum_{i=0}^{b-1}w_i a_i.
\end{align*}
The $c_i$ decrease, and $c_i\le w_i+tT_H(t)$.
Complete the root on its other side as in the terminal proof, so
its two cuts have sizes $d-p,p$. Apply
\Cref{lem:even-certificate}. If $a$ is the first index of $I$,
\[
 \Delta_I\le R_I\{w_a+tT_H(t)\}.
\]
Since $w_i\le w_a$ on $I$, it follows that
\begin{equation}\label{eq:even-weightedsum}
 T_H(t)\le\sum_I w_a^2 R_I
              +tT_H(t)\sum_I w_aR_I.
\end{equation}
The last sum is $O_d(b^2)$, since $R_I=O_d(1)$ and there are at most
$b$ intervals. The second term is therefore $O_d(tb^5)$, by
\eqref{eq:even-responsebound}.

For the first term use \eqref{eq:even-certificate}. Bounded interval
lengths give
\[
 \frac1\beta\sum_I |I|w_a^2
 =\frac1\beta\sum_{i=0}^{b-1}w_i^2+O_d(b^2)
 =\frac{b^3}{3\beta}+O_d(b^2).
\]
The sum $\sum_Iw_a^2(\psi_I^- -\psi_I^+)$ is also $O_d(b^2)$:
collect the terms at each common endpoint and use the boundedness
of $\psi$ and the total variation, at most $b^2$, of the decreasing
sequence $w_a^2$. The two outer terms have the same bound.
Finally, \Cref{lem:even-certificate} gives at least
$c_d\lfloor\xi b\rfloor-C_d$ strict intervals in the stated prefix.
At their first indices $w_a\ge b/2$. Their total saving is therefore
at least $c_{d,\xi}b^3-C_{d,\xi}b^2$. Substitute these estimates in
\eqref{eq:even-weightedsum} to obtain \eqref{eq:even-endloss}.
\end{proof}

\subsection{Regular comparison subgraphs at the same order}
\begin{lemma}\label{lem:even-comparisons}
For each $k,l\in\{2,d-2\}$ and every sufficiently large $m$,
there is an admissible two-terminal graph $J_m$ of type $(k,l)$ with
\begin{equation}\label{eq:even-comparisonresistance}
 R_{J_m}\ge\frac{m-1}{\beta}-C_d.
\end{equation}
For each $p\in\{2,d-2\}$ and every sufficiently large $b$, there is
an admissible end subgraph $E_b$ with root degree $p$ and $b$
nonroot vertices such that
\begin{equation}\label{eq:even-comparisonend}
 T_{E_b}(0)\ge\frac{b^3}{3\beta}-C_db^2.
\end{equation}
\end{lemma}
\begin{proof}
Use the graphs $B_{N,p}$ from \Cref{lem:comparison-completions}. Recall that
the root has degree $p$, all other degrees are $d$, and deleting any
one remaining edge of difference one leaves the graph connected.

Take disjoint $B_{N,k}$ and $B_{D,d-l}$. In each choose a remaining
difference-one edge. Delete these two edges and replace them by
two edges between the two subgraphs, pairing their endpoints.
The graph is simple and connected, and every degree is unchanged.
Its two roots have the required degrees. Its order is $N+D+2$;
as $N$ runs through $D,\ldots,2D-1$, these orders cover every residue
modulo $D$.
Prepend copies of $C_k$, identifying each right singleton with the
next left singleton, and finally with the first root. At each
identification the degrees are $d-k$ and $k$, so the shared vertex
has degree $d$. Each copy adds $D$ vertices and coefficient $D/\beta$.
The remaining graph has bounded order, proving
\eqref{eq:even-comparisonresistance} for every sufficiently large $m$.

For the end subgraph write $b=qD+N$ with $D\le N\le2D-1$ and
$q\ge0$. Join $q$ copies of $C_p$ to $B_{N,p}$ in the same way.
There are exactly $b$ vertices other than the first root, and their
degrees are $d$. Set the root value to zero and solve $A\bu=\1$.
Symmetry, or subtracting equations, makes $\bu$ constant on each
of the two middle cliques of a copy of $C_p$.
If $I$ vertices remain beyond its first root, the currents on its
three interfaces are $I$, $I-p$, and $I-d$. Thus its contribution
to $\bu^\top A\bu=T_{E_b}(0)$ is
\[
 \frac{I^2}{p}+\frac{(I-p)^2}{p(d-p)}
                +\frac{(I-d)^2}{d-p}
 =\frac D\beta I^2+O_d(I+1).
\]
Here $I=b,b-D,\ldots,b-(q-1)D$. The sum of the leading terms is
$b^3/(3\beta)+O_d(b^2)$, the sum of the errors is $O_d(b^2)$,
and the last bounded subgraph contributes a nonnegative quantity.
This proves \eqref{eq:even-comparisonend}.
\end{proof}

All the replacements just constructed have exactly the same number
of vertices as the subgraphs they replace. Their specified root
degrees are also identical. Thus attaching them to the unchanged
outside subgraphs preserves simplicity, connectedness and every degree.
There is no parity restriction on the order in this construction.

\subsection{Every exceptional interval has bounded order}
In this subsection, we use the comparison subgraphs to rule out
exceptional intervals whose order tends to infinity.

Fix an increasing unit Fiedler vector of an exact minimizer $G$.
Select all its standard gaps. Their difference intervals have disjoint
interiors, and consecutive ones may share a singleton vertex.
The proof of \Cref{thm:evenmain}, together with
\eqref{eq:comparisonerror}, shows that all but $O_d(n^{2/3})$
difference indices belong to these gaps.
The remaining nonempty intervals between standard gaps, and the two
intervals at the ends, will be called \emph{exceptional intervals}.
We include both endpoint vertices in an interior interval, and its
attachment root in an end interval. They contain no standard gap.
Their total number of vertices, with these boundary vertices counted,
is $O_d(n^{2/3})$. Indeed, each nonempty interval accounts for at
least one omitted difference, and including its endpoints adds at
most twice the number of intervals.

An interior exceptional interval induces a connected admissible
subgraph $H$ of some type $(k,l)$, with $k,l\in\{2,d-2\}$.
Only its two separating endpoints meet the rest of $G$.
An end interval is an admissible end subgraph. Connectedness follows
from that of $G$ and the fact that no endpoint can be bypassed.
The two values at the ends of an interior interval satisfy $a<b$.
Otherwise all its values would be equal, and the current across its
first cut would be zero. But for every proper prefix of an increasing
mean-zero nonconstant vector, $\sum_{v\le i}x_v<0$; hence
\eqref{eq:even-b} makes this current positive.

\begin{lemma}\label{lem:even-sign-count}
For the increasing unit Fiedler vector of an exact minimizer, at least
$c_dn$ components have each strict sign. In particular, every initial
or final interval of $o(n)$ vertices has one strict sign, including
its attachment root.
\end{lemma}
\begin{proof}
Since $\mu(G)=O_d(n^{-2})$,
\Cref{lem:smallentries} gives $\max|x_v|\le C_d/\sqrt n$.
Therefore $\sum|x_v|\ge\sqrt n/C_d$. Each strict sign has total
absolute sum at least $\sqrt n/(2C_d)$, and so occurs at least
$c_dn$ times. In particular, any initial or final interval of
$o(n)$ vertices has one strict sign, including its attachment root.
\end{proof}

\begin{lemma}\label{lem:even-interiorloss}
Let $H$ be an interior exceptional interval with $m$ vertices,
endpoint values $a<b$, $\Delta=b-a$, and $M=\max\{|a|,|b|\}$.
For sufficiently large $m$,
\begin{equation}\label{eq:even-interiorloss}
 \mathcal E_H(\x)\ge\frac{\beta+\eta_d}{m}\Delta^2-C_d\mu M\Delta,
\end{equation}
where $\eta_d>0$.
If at least $t$ vertices lie on each outside side of $H$, then also
\begin{equation}\label{eq:even-currentlower}
 \mathcal E_H(\x)\ge\mu(t-m)M\Delta.
\end{equation}
\end{lemma}
\begin{proof}
Apply \Cref{lem:even-certificate} to the inherited ordering of
$H$. On each interval $I$, \eqref{eq:even-b} shows that the variation
of $c_i$ is at most $C_d\mu M$. Hence, writing
$E_I=\sum_{i\in I}a_i c_i$, we have
\[
 E_I\ge\Delta_I\max_{i\in I}c_i-C_d\mu M\Delta_I
     \ge\frac{\Delta_I^2}{R_I}-C_d\mu M\Delta_I.
\]
Every internal cut sees only edges of $H$, so
$\sum_I E_I=\mathcal E_H(\x)$. Since $H$ contains no standard gap,
at least $c_dm-C_d$ intervals are strict. Telescoping gives
\[
 \sum_I R_I\le\frac{m-1}{\beta}+\phi(k)-\phi(l)-c'_dm+C_d.
\]
Decrease $c'_d$ if necessary so that $c'_d<1/(2\beta)$. For large
$m$ the right side is at most $m/(\beta+\eta_d)$ for some
$\eta_d>0$. Cauchy--Schwarz applied to $\sum\Delta_I^2/R_I$
proves \eqref{eq:even-interiorloss}.

If $M=-a$, then $a<0$. Every outside vertex to the left has component
at most $-M$, while the contribution of the vertices of $H$ to a
prefix is at most $mM$. Thus \eqref{eq:even-b} gives
$c_i\ge\mu(t-m)M$ at every cut inside $H$.
If $M=b>0$, use the corresponding suffix to get the same bound.
Multiply by $a_i\ge0$ and sum. This proves
\eqref{eq:even-currentlower}.
\end{proof}

\begin{theorem}\label{thm:even-boundedintervals}
For fixed even $d\ge6$, the orders of all exceptional intervals of
sufficiently large exact minimizers are bounded in terms of $d$.
\end{theorem}
\begin{proof}
Suppose instead that an exceptional interval has order $m\to\infty$
along a sequence of minimizers. We already know $m=o(n)$ and
$\mu=O_d(n^{-2})$.

First let $H$ be interior and at least $t$ vertices lie on each
outside side. Suppose $t\ge Km$, where $K>1$ will be fixed below.
By \eqref{eq:even-currentlower}, the error term in
\eqref{eq:even-interiorloss} is at most
$C_d\mathcal E_H(\x)/(t-m)$. For large $m$ it follows that
\[
 \mathcal E_H(\x)\ge\frac{\beta+\eta_d/2}{m}\Delta^2.
\]
Replace $H$ by $J_m$ of the same type from
\Cref{lem:even-comparisons}, and use its minimizing vector
with the same endpoint values. For large $m$ its quadratic form
is at most $(\beta+\eta_d/4)\Delta^2/m$. It therefore saves at
least $\chi_d\mathcal E_H(\x)$ for some fixed $\chi_d>0$.

All old and new values on the changed vertices lie in $[a,b]$.
The loss of squared norm is at most $2mM\Delta$, and the change
of the sum is at most $m\Delta$. Centering loses at most a further
$m^2\Delta^2/n\le2mM\Delta$. Thus the total loss of centered
squared norm is at most $4mM\Delta$.
Choose $K>1+4/\chi_d$. By \eqref{eq:even-currentlower}, the saving
in the quadratic form exceeds $\mu$ times that possible loss.
This contradicts exact minimality.

We are left with $t<Km$, or with an exceptional interval meeting
an end. In the first case include the whole nearer end up to the
far endpoint of $H$ and take that endpoint as root. In the second
case use the end interval itself. After possibly reversing the
ordering, this is an admissible end subgraph with $b$ nonroot vertices,
where
$m-1\le b\le(K+1)m$ and $b=o(n)$.
The first $m-1$ differences from its root belong to the exceptional
interval and contain no standard gap. Its root value $a_0$ is
nonzero, by the sign count. The eigen-equations give
\[
 \x_{H-r}=a_0\{\1+\mu(A_H-\mu I)^{-1}\1\}.
\]
Thus ordering the vector $(A_H-\mu I)^{-1}\1$ from the root is
exactly the inherited ordering, reversed when $a_0<0$. Keep ties
in that order. For the fixed $\xi=1/[4(K+1)]$, the first
$\lfloor\xi b\rfloor$ differences are still in the exceptional
interval when $m$ is large.

Since $\mu b^2=o(1)$, Lemmas~\ref{lem:even-endloss} and
\ref{lem:even-comparisons} imply
\[
 T_H(\mu)<T_{E_b}(0)\le T_{E_b}(\mu)
\]
for large $m$. The last inequality is the monotonicity in
\Cref{lem:even-response}. Since $\mu b^2=o(1)$,
\Cref{lem:even-replaceend} gives a strict improvement at the
same order and with all degrees unchanged, a contradiction.
No sequence of unbounded exceptional intervals exists, proving
the uniform bound.
\end{proof}

\subsection{Interchanging a bounded interval and a standard block}
In this subsection, we compare moving a bounded exceptional interval
one block to the left or right. The preceding theorem lets us use estimates uniform over a
finite family of admissible subgraphs. Shared vertices require a
small change from the bridge calculation: give each terminal half
weight in the squared norm of a subgraph. At a common terminal the
two halves then add to its original weight one. No graph or degree
is changed by this convention.
Explicitly, for terminals $s,t$, put
\[
 \|\z\|_*^2=\sum_{v\ne s,t}z_v^2+\frac{z_s^2+z_t^2}{2}.
\]

For an admissible $H$ of type $(k,l)$ and order $m$, put
$\nu_H=m-1$, $r_H=R_H$, $r_C=D/\beta$, and
\begin{equation}\label{eq:even-defect}
 \delta_H=\nu_Hr_C-D\{r_H+\phi(l)-\phi(k)\}.
\end{equation}
\Cref{thm:even-terminal} gives $\delta_H\ge0$, with equality
only for a standard chain. Let $z$ be the unit-current vector on
$H$, with endpoint values $0,r_H$. Put
\begin{align}
 A_H^*&=\sum_{v\ne s,t}z_v+\frac{r_H}{2},\label{eq:even-moment}\\
 S_H&=A_H^*-\frac{\nu_Hr_H}{2},\notag\\
 W_H&=S_H-\frac{\beta r_H}{2}\{\phi(k)+\phi(l)\}.\notag
\end{align}
The star distinguishes this scalar from the matrix used for an end
subgraph.

\begin{lemma}\label{lem:even-interchange}
Join $H$ to $C_l$, or $C_k$ to $H$, by identifying the adjoining
terminals. These are admissible subgraphs of the same order and type
$(k,l)$. For fixed outside values $a,b$, let $q_{HC_l}(t;a,b)$ be the
minimum of the quadratic form minus $t$ times the squared norm,
with half weight at the two outside terminals. Define $q_{C_kH}$
in the same way. For sufficiently small $t\ge0$, put
$r=r_H+r_C$, $c=(a+b)/2$, and $J=(b-a)/r$.
Then
\begin{equation}\label{eq:even-interchange}
 q_{HC_l}-q_{C_kH}
 =2t\delta_HcJ+2tr_CW_HJ^2+O_{H,d}(t^2(a^2+b^2)).
\end{equation}
For a fixed bound on $m$, the remainder and its validity interval
are uniform over all admissible graphs and endpoint choices.
\end{lemma}
\begin{proof}
All unfixed vertices in either order have weight one. Their matrix
$B$ with the two outside values fixed is positive definite.
Let $f_0$ be the full minimizing vector at $t=0$. In
\eqref{eq:boundary-expansion}, take the free vector to be
$(f_0)_{\rm int}$ and $\omega=(a^2+b^2)/2$.
The first-order term is therefore $-t\|f_0\|_*^2$, including
the two fixed half-weights. For $0\le t\le\lambda_{\min}(B)/2$,
the inverse has norm at most $2/\lambda_{\min}(B)$ and
$\|(f_0)_{\rm int}\|^2=O_{H,d}(a^2+b^2)$.
Thus the last term is $O_{H,d}(t^2(a^2+b^2))$, uniformly over a
finite family.

Here is the calculation of the first-order term, including the
orientation change. For the unit-current vector of $C_p$, the
internal values are $1/p$ and $(d-p+1)/[p(d-p)]$. Its half-weighted
mass is $D$, its first coordinate sum is
\begin{equation}\label{eq:even-cellmoment}
 A_C(p)=\frac Dp+\frac{r_C}{2}
       =\frac{Dr_C}{2}+D\phi(p),
\end{equation}
and its second coordinate sum, denoted $B_C(p)$, satisfies
\begin{equation}\label{eq:even-cellsecond}
 B_C(l)-B_C(k)=Dr_C\{\phi(l)-\phi(k)\}.
\end{equation}
Indeed, there are only the two orientations; reversal sends every
coordinate $z$ to $r_C-z$, proving \eqref{eq:even-cellsecond}
from \eqref{eq:even-cellmoment}.

At unit current the first coordinate sums in $HC_l$ and $C_kH$
are respectively $A_H^*+A_C(l)+Dr_H$ and
$A_C(k)+A_H^*+\nu_Hr_C$. Their difference is $-\delta_H$.
For the second coordinate sums the difference is
\[
 Dr_H^2+2r_HA_C(l)+B_C(l)
       -\nu_Hr_C^2-2r_CA_H^*-B_C(k)
 =-r\delta_H-2r_CW_H,
\]
by \eqref{eq:even-cellmoment}--\eqref{eq:even-cellsecond}.
For boundary values $a,b$, the harmonic vector is $a\1+Jz$.
The difference of its squared norms is therefore
$-2\delta_HcJ-2r_CW_HJ^2$. The zero-parameter quadratic forms in
both orders are $(b-a)^2/r$. Substitution proves
\eqref{eq:even-interchange}.
\end{proof}

Notice why one cannot simply interchange the same oriented block
with $H$ when $k\ne l$. The correct alternatives are $C_kH$ and
$HC_l$. Their common vertex has degree $d$ in both cases, and the
two outside terminal degrees are unchanged. Formula
\eqref{eq:even-defect} is exactly the correction required by this
change of orientation.

\begin{lemma}\label{lem:even-movetoend}
Fix $B$. In a sufficiently large exact minimizer, any interior
exceptional interval of order at most $B$, immediately preceded and
followed by standard blocks, has an outside side with at most
$T_{d,B}$ vertices.
\end{lemma}
\begin{proof}
First, such an interval $H$ cannot be a standard chain in a different
ordering. Values on each middle clique of a standard block coincide,
by subtracting their eigen-equations. If both endpoint values have
the same weak sign, then all values on $H$ have that sign. The
currents along its natural interfaces form a monotone sequence,
by the eigen-equations. Its first and last currents are positive,
as follows from \eqref{eq:even-b} and the ordered separating endpoints.
Thus every natural interface has a positive difference.
If the endpoint values have opposite strict signs, write
$\Delta=b-a$. Their largest absolute value is at most $\Delta$.
With the endpoints fixed, the harmonic vector $f_0$ and the Fiedler
restriction $f$ satisfy
\begin{equation}\label{eq:even-harmonicerror}
 B(f-f_0)=\mu f_{\rm int}.
\end{equation}
For bounded order, $B^{-1}$ is bounded by the same path argument
as in \Cref{lem:even-response}. Therefore $f-f_0=O_{d,B}(\mu\Delta)$.
Every distinct consecutive group in the harmonic chain differs by
at least a fixed positive multiple of $\Delta$. These differences
remain positive for small $\mu$. In either case, its natural order
agrees with the Fiedler order apart from ties within cliques, and
it contains standard gaps, a contradiction.

It follows from \Cref{thm:even-terminal} that $\delta_H>0$.
There are only finitely many possible $H$ of bounded order, so
$\delta_H$ has a uniform positive lower bound. All constants below
are uniform over this family.

The three consecutive subgraphs are $C_k,H,C_l$. Let $a_H,b_H$
be the endpoint values of $H$, and put
$c_H=(a_H+b_H)/2$ and $I=(b_H-a_H)/r_H>0$.
Let $M$ bound the absolute values on all three subgraphs.
Equation~\eqref{eq:even-harmonicerror} gives an $O(\mu M)$ error
from the harmonic vector on each bounded subgraph, including an
$O(\mu M)$ error in the endpoint currents. At a common terminal,
the incoming and outgoing currents differ by $\mu$ times its
component. Therefore the neighboring block currents are
$I+O(\mu M)$. For the union on the left and the union on the right,
the corresponding parameters in \Cref{lem:even-interchange}
satisfy
\begin{align*}
 c_L&=c_H-r_CI/2+O(\mu M),\\
 J_L&=I+O(\mu M),\\
 c_R&=c_H+r_CI/2+O(\mu M),\\
 J_R&=I+O(\mu M).
\end{align*}
Also $M\le C(|c_H|+I)$ for small $\mu$: first obtain this inequality
with an additional $C\mu M$ on the right and absorb it.

Keep the outside values fixed and consider the two replacements
\begin{align*}
 C_kH&\longmapsto HC_l,\\
 HC_l&\longmapsto C_kH.
\end{align*}
Each preserves the order and all degrees. In either pair the old
restriction minimizes its quadratic form minus $\mu$ times the
squared norm, since its Dirichlet matrix is positive definite for
small $\mu$. The sum of the two changes, before centering, is
\begin{equation}\label{eq:even-twoswaps}
 -2\mu\delta_Hr_CI^2+O(\mu^2M^2).
\end{equation}
This follows by applying \eqref{eq:even-interchange} in the two
directions. The $c_H$ and $W_H$ terms cancel.
Only a bounded number of actual vertex values change, each by
$O(I+\mu M)$. Consequently the change in their sum is of that
same order. Since the original vector has mean zero, centering
adds at most
\begin{equation}\label{eq:even-swapcenter}
 C\mu(I+\mu M)^2/n
\end{equation}
to either comparison. Shared vertices are counted once: the two
half weights at each internal shared vertex add to one, and the
outer endpoints have not changed.

Choose a large constant $K$. If $|c_H|\le KI$, then $M=O_K(I)$.
The sum \eqref{eq:even-twoswaps}, with both centering terms, is
negative for large $n$. At least one replacement improves the
Rayleigh quotient.

Suppose $c_H>KI$; the negative case follows by reversal and a sign
change. Let $t$ be the smaller number of vertices outside the three
subgraphs. For large $K$, the last singleton has value at least
$c_H/2$, and every vertex after it has at least that value.
The outgoing current there is consequently at least $\mu t c_H/2$,
by a suffix sum. It is $I+O(\mu M)$ by the eigen-equation at that
singleton and the preceding harmonic estimates. Since $M=O(c_H)$,
\[
 \frac{\mu M}{I}\le\frac Ct
\]
when $t$ exceeds a fixed constant.
For the right replacement, \eqref{eq:even-interchange} has leading
term $-2\mu\delta_Hc_RJ_R$. Choose $K$ large enough to absorb the
term involving $W_HJ_R^2$, and then choose $t$ large enough to
absorb $O(\mu^2M^2)$. The change is at most $-c\mu c_HI<0$.
The centering term \eqref{eq:even-swapcenter} is smaller for large
$n$, again giving a contradiction.
Thus $t$ is bounded. Adding the bounded orders of the neighboring
blocks gives the assertion about an outside side of $H$.
\end{proof}

\begin{theorem}\label{thm:even-boundedends}
Every sufficiently large exact minimizer consists of an uninterrupted
chain of copies of $M_d$ and two connected end subgraphs with
$O_d(1)$ vertices in total. Consecutive copies share their singleton
endpoints with consistent orientation, and each end subgraph meets
the chain only at its outside singleton.
\end{theorem}
\begin{proof}
\Cref{thm:even-boundedintervals} bounds every exceptional
interval, including those at the ends. Every interior exceptional
interval has a standard block immediately on both sides, so
\Cref{lem:even-movetoend} puts it within a bounded number of
vertices of an end. Thus every exceptional vertex belongs to an
initial or final bounded segment of the increasing order. Extend
these two segments, if necessary, to adjacent separating singletons;
this adds at most $O_d(1)$ vertices. Between them all gaps are
standard. Sharing a singleton forces their orientations to agree,
since $2+2$ and $2(d-2)$ are different from $d$ for $d\ge6$.
Increasing the order threshold ensures that at least two middle
blocks remain.
\end{proof}

\subsection{The end subgraphs do not branch}
In this subsection, we exclude branching inside the bounded end
subgraphs. Together with the middle chain, this proves that the whole
block-tree is a path. All vertices in either bounded
end subgraph have one strict Fiedler sign, by the sign count above.
Change the sign when needed to make the end under consideration
positive.

\begin{lemma}\label{lem:even-outward-component}
Let $U$ be a component of $G-v$ lying in the bounded positive end
under consideration, and suppose that $x_v>0$. For sufficiently
large $n$,
\begin{equation}\label{eq:even-outward}
 \x_U-x_v\1=\mu x_v(L_U-\mu I)^{-1}\1>\0,
\end{equation}
where $L_U$ is the principal Laplacian on $U$.
\end{lemma}
\begin{proof}
Paths to $v$ give $\lambda_{\min}(L_U)\ge |U|^{-2}$.
For large $n$, the positivity argument in \Cref{lem:even-response}
applies to $L_U-\mu I$. The eigen-equations then give
\eqref{eq:even-outward}.
\end{proof}

\begin{lemma}\label{lem:even-cutbranch}
Deleting a vertex of a sufficiently large exact minimizer leaves
at most two components.
\end{lemma}
\begin{proof}
Only a vertex $v$ in a bounded end could violate the assertion.
Suppose that two components $A,B$ of $G-v$ lie away from the middle
chain. Their values exceed $x_v>0$ by \eqref{eq:even-outward}.
Put $F_A=\sum_Ax$, $F_B=\sum_Bx$, $p=|N(v)\cap A|$, and
$q=|N(v)\cap B|$. Relabel so that $F_A/p\le F_B/q$.
Choose $u\in N(v)\cap A$ with smallest value and put
$\eta=x_u-x_v>0$. Summing the equations over $A$ gives
\begin{align*}
 p\eta&\le\mu F_A,\\
 q\eta&\le\mu F_B.
\end{align*}
Let $Z=\{w\in N(u)\cap A:vw\notin E(G)\}$ and $h=|Z|$.
There are at most $p-1$ common neighbors of $u,v$ in $A$, so
$h\ge d-p\ge q+1\ge2$. The second inequality uses an edge from
$v$ to a third component. Every common neighbor in $A$ has value
at least $x_u$, by the choice of $u$. Its eigen-equation therefore
gives
\[
 \sum_{w\in Z}(x_w-x_u)\le\eta-\mu x_u.
\]
Take the $q$ smallest values in $Z$, forming $W$. Their sum of
differences is at most $q(\eta-\mu x_u)/h$.

Move all edges from $v$ to $B$ to $u$, and move the edges $uw$
for $w\in W$ to $v$. Every degree is unchanged, and all added edges
are new. The edge $uv$ remains. Each component detached by deleting
edges $uw$ contains an endpoint $w$ and is reattached by $vw$;
the moved component $B$ remains attached at $u$. Thus the graph
is connected. Add $\eta$ to every component on $B$, keeping all
other values. The moved edges to $B$ keep their differences.
The change of the quadratic form minus $\mu$ times centered
squared norm is at most
\[
 -q\eta^2(1-2/h)-2q\mu x_u\eta/h
       -\mu|B|(1-|B|/n)\eta^2<0.
\]
Here the change in variance is
$2\eta F_B+|B|(1-|B|/n)\eta^2$, and the quadratic-form change is
$q\eta^2+2\eta\sum_{w\in W}(x_w-x_u)$.
The Rayleigh principle gives a contradiction.
\end{proof}

\begin{lemma}\label{lem:even-levelswitch}
Let $A$ be a component of $G-a$ in a bounded positive end, lying
away from the middle chain. Choose $w\in A$ with largest component.
Then no edge $uv$ outside $A\cup\{a\}$ with $x_u<x_v$ satisfies
$x_u\le x_w\le x_v$.
\end{lemma}
\begin{proof}
First $G-w$ is connected. Otherwise a component $U$ not containing
$a$ lies in $A\setminus\{w\}$; \eqref{eq:even-outward} would make
all its values exceed $x_w$, a contradiction. It follows also that
$G[A\cup\{a\}]-w$ is connected: a path leaving this subgraph can
leave and return only at $a$, and that excursion can be removed.
Choose a neighbor $z\in A$ of $w$. Such a neighbor exists, and
$w$ has another neighbor besides $z$, because only $a$ can be its
neighbor outside $A$ and $d\ge3$. Therefore deleting $wz$ leaves
$G[A\cup\{a\}]$ connected.

Suppose the stated edge $uv$ exists. Replace $wz,uv$ by $wv,zu$.
Both added edges are new and every degree is preserved. If deleting
$uv$ disconnects $G-A$, the two new edges join both resulting
components through the connected subgraph containing $w,z,a$.
Thus the whole graph remains connected.
The quadratic-form change on the unchanged vector is
\[
 2(x_w-x_u)(x_z-x_v)\le0,
\]
since $x_z\le x_w$. Strict inequality contradicts minimality.
If $x_w=x_u$, then $x_z-x_v\ne0$, so the eigen-equation at $w$
changes. The other equality possibility is $x_z=x_v$, which
forces $x_z=x_w=x_v>x_u$; the eigen-equation at $z$ then changes.
In either case the old vector cannot attain equality in the new
Rayleigh principle, again giving strict improvement.
\end{proof}

\begin{lemma}\label{lem:even-blockbranch}
No block in a bounded end has two distinct outward cut vertices.
\end{lemma}
\begin{proof}
Root the block-tree at its attachment to the middle chain.
Suppose a block $B$ has distinct outward cut vertices $a,b$, with
outward components $A,C$. By \Cref{lem:even-cutbranch}, these
are their only outward components. Equation~\eqref{eq:even-outward}
gives $\x_A>x_a\1$ and $\x_C>x_b\1$.
Write $M_A=\max_Ax$, $M_C=\max_Cx$, and relabel so that $M_A\le M_C$.
If $x_b\le M_A$, a path from $b$ to a maximum vertex of $C$ has a
nonconstant edge with endpoint values enclosing $M_A$. This includes
$x_b=M_A$, since $M_C>x_b$. It contradicts
\Cref{lem:even-levelswitch} for a maximum vertex in $A$.
Hence $x_b>M_A$.

Since $B-a$ is connected, a vertex there with value at most $M_A$
would yield the same contradiction by a path to $b$.
Thus every vertex of $B-a$ has value greater than $M_A>x_a$.
All other neighbors of $a$ lie in $A$, by
\Cref{lem:even-cutbranch}, and have value greater than $x_a$.
This contradicts
$\sum_{v\sim a}(x_a-x_v)=\mu x_a>0$.
\end{proof}

\begin{proof}[Proof of \Cref{thm:evenfull}]
For $d\ge6$, \Cref{thm:even-boundedends} supplies the chain
and bounded end subgraphs. \Cref{lem:even-cutbranch} excludes
branching at a cut-vertex node. \Cref{lem:even-blockbranch}
excludes branching at a block-node in an end. The middle blocks
already have just their two prescribed attachments. Hence the entire
block-tree is a path. This proves the stated structural assertion.

If the middle contains $q$ blocks, then $n=qD+O_d(1)$. Each adds
three to the distance between the outside singletons, so
$\dm(G)\ge3q=3n/D-O_d(1)$. The general upper bound
$\dm(G)<3n/D$ from \cite{Caccetta,Erdos} gives the equality.
Together with \cite[Theorems~5.1--5.2]{AbGhDiam}, this also shows that
the diameter differs from the maximum by at most $O_d(1)$. For $d=4$, the published structural theorem
\cite[Theorem~3.2]{AbGhQuartic} gives the same conclusions, with a finite
list of end blocks. No extension to near-minimizing quartic graphs
is used in this deduction.
\end{proof}

\section{The minimum algebraic connectivity}\label{sec:chainvalue}
We now deduce the numerical formula from the structure of exact
minimizers, in the same order as in Part~I. The calculation is the
chain-to-path comparison of \cite[Section~3.2]{AbGhDiam}, with the
bounded-end errors retained. It applies to $d=4$ as well.

\begin{lemma}\label{lem:even-chainvalue}
Fix an even $d\ge4$ and an integer $B\ge0$. Suppose $G$ consists
of $q$ consistently oriented $M_d$ blocks and two connected end
subgraphs. Each end subgraph meets the chain only at the corresponding
outer singleton, there are no other edges between the pieces, and
at most $B$ vertices lie outside the chain. If $n=|V(G)|$, then
\begin{equation}\label{eq:even-chainvalue}
 \mu(G)=\frac{2(d-2)\pi^2}{n^2}+O_{d,B}(n^{-3})
 \qquad(q\to\infty).
\end{equation}
The error is uniform over the choices of the two end subgraphs.
\end{lemma}
\begin{proof}
Put $D=d+1$, $\beta=2(d-2)$, and list the singleton vertices as
$s_0,\ldots,s_q$. If $b$ vertices lie outside the chain, then
\[
 n=qD+1+b,\qquad 0\le b\le B.
\]
The cosine trial vector from \Cref{lem:comparison-upperbound} applies
here too: give $s_j$ value $-\cos(\pi j/q)$, minimize inside each
block, and extend constantly over each end subgraph. Internal edges
of an end subgraph contribute zero. The energy and centered squared
norm are respectively
\[
 \frac{\beta\pi^2}{2Dq}+O_d(q^{-3}),
 \qquad \frac n2+O_{d,B}(1).
\]
The mean is $O_{d,B}(n^{-1})$, by the same bounded-variation sum
used there. Hence
\begin{equation}\label{eq:even-chain-upper}
 \mu(G)\le\frac{\beta\pi^2}{n^2}+O_{d,B}(n^{-3}).
\end{equation}
In particular, $\mu=O_{d,B}(n^{-2})$.

Choose a mean-zero unit Fiedler vector $\x$ of $G$ and put
$u_j=x_{s_j}$. No monotonicity in the chain order is needed.
By \Cref{lem:smallentries},
\begin{equation}\label{eq:even-chain-errors}
 \max_v|x_v|=O_{d,B}(n^{-1/2}),\qquad
 |x_v-x_w|=O_{d,B}(n^{-3/2})\quad(vw\in E(G)).
\end{equation}
Assign to block $j$ all its vertices except its right singleton;
these are $D$ vertices, at distance at most two from $s_{j-1}$.
The assigned sets are disjoint. Their component sums differ from
$Du_{j-1}$ by $O_{d,B}(n^{-3/2})$, and their squared component
sums differ from $Du_{j-1}^2$ by $O_{d,B}(n^{-2})$.
Summing over the $q=O(n)$ blocks gives total errors
$O_{d,B}(n^{-1/2})$ and $O_{d,B}(n^{-1})$, respectively.
The unassigned vertices, namely $s_q$ and the vertices outside the
chain, have the same total error bounds by \eqref{eq:even-chain-errors}.
Adding the one extra sample $u_q$ therefore yields
\begin{align}
 D\sum_{j=0}^{q}u_j&=O_{d,B}(n^{-1/2}),
 \label{eq:even-chain-mean}\\
 D\sum_{j=0}^{q}u_j^2&=1+O_{d,B}(n^{-1}).
 \label{eq:even-chain-mass}
\end{align}
For $\bar u=(q+1)^{-1}\sum_{j=0}^q u_j$, the mean correction is
$(q+1)\bar u^2=O_{d,B}(n^{-2})$. Thus
\begin{equation}\label{eq:even-chain-centered}
 \|\bu-\bar u\1\|^2=\frac1D+O_{d,B}(n^{-1}).
\end{equation}

Average the vector on the two middle cliques of each block and use
\Cref{lem:comparison-blockminimum}. The edge sets of different
blocks are disjoint, even though consecutive blocks share a vertex.
All omitted end-subgraph terms are nonnegative, so
\begin{equation}\label{eq:even-chain-path}
 \mu=\x^\top L(G)\x
 \ge\frac{\beta}{D}\sum_{j=0}^{q-1}(u_{j+1}-u_j)^2
 \ge\frac{\beta}{D}p_{q+1}\|\bu-\bar u\1\|^2.
\end{equation}
Since $D(q+1)=n+O_{d,B}(1)$ and
$p_{q+1}=\pi^2/(q+1)^2+O(q^{-4})$, substitution of
\eqref{eq:even-chain-centered} gives
\[
 \mu\ge\frac{\beta\pi^2}{n^2}-O_{d,B}(n^{-3}).
\]
Together with \eqref{eq:even-chain-upper}, this proves the lemma.
All estimates depend only on $d,B$, not on the internal form of the
bounded end subgraphs.
\end{proof}

\begin{proof}[Proof of \Cref{thm:evenminimum}]
Choose a graph attaining $m_d(n)$. For $d\ge6$,
\Cref{thm:evenfull} gives an uninterrupted $M_d$ chain and two
connected end subgraphs with at most $B_d$ vertices in total.
For $d=4$, the published characterization
\cite[Theorem~3.2]{AbGhQuartic} gives the same chain with bounded
end subgraphs. Apply \Cref{lem:even-chainvalue}, with $B$ depending
only on $d$, to obtain \eqref{eq:even-minimum} in every stated degree.
\end{proof}

\begin{corollary}\label{cor:even-lower}
For each fixed even $d\ge4$, every sufficiently large $d$-regular
graph satisfies
\begin{equation}\label{eq:even-sharp}
 \mu(G)\ge\frac{2(d-2)\pi^2}{n^2}-\frac{C_d}{n^3}.
\end{equation}
\end{corollary}
\begin{proof}
Every graph in the class has $\mu(G)\ge m_d(n)$, so the preceding
corollary gives the result. For $d\ge6$, the path comparison also
proves this directly, without exact minimality: if
$\mu\le(\beta\pi^2+1)/n^2$, then
\eqref{eq:even-discrete-energy}, the path inequality, and
\eqref{eq:even-norm} give
\[
 \mu\ge\beta p_n\|\y-\bar y\1\|^2-C_dn^{-3}
       =\frac{\beta\pi^2}{n^2}-O_d(n^{-3}).
\]
If that upper condition on $\mu$ fails, the desired inequality is
immediate. This records the numerical consequence of the local
comparison separately from its use in the structural proof.
\end{proof}

\section{Algebraic connectivity and maximum diameter}
\label{sec:diameter}
In this section, we prove \Cref{cor:diameter} by counting the distance
contributed by the $M_d$ blocks along the block-tree.

\begin{proof}[Proof of \Cref{cor:diameter}]
The proof of \Cref{thm:evenmain} retains $k$ original
intervals of length $d+1$ inducing $M_d$ blocks, with
\[
 (d+1)k\ge n-O_d((\e+n^{-1})^{1/3}n).
\]
The two singleton endpoints of each block are separating vertices,
and the selected blocks occur along one ordered path. A path
between the outside endpoints must therefore traverse every
selected block between its singleton endpoints, whose distance
inside the block is three. Hence $\dm(G)\ge3k$, and under the
hypothesis of \Cref{thm:evenmain} we obtain
\begin{equation}\label{eq:even-diameter-quant}
 \dm(G)\ge\frac{3n}{d+1}-C_d(\e+n^{-1})^{1/3}n.
\end{equation}
Exceptions between selected blocks
cannot provide shortcuts past their separating endpoints.

\Cref{thm:evenminimum} gives
$n^2m_d(n)\to2(d-2)\pi^2$. If $\mu(G_n)/m_d(n)\to1$, put
\[
 \e_n=\max\{0,n^2\mu(G_n)-2(d-2)\pi^2\}.
\]
Then $\e_n\to0$, and for all large $n$ it lies in $[0,1]$.
The quantitative diameter bound and the general upper bound
$\dm(G_n)<3n/(d+1)$ from \cite{Caccetta,Erdos} prove
\eqref{eq:diameterimplication}. For an exact minimizer,
\Cref{thm:evenfull} improves this to
$\dm(G)=3n/(d+1)+O_d(1)$, as stated in
\eqref{eq:even-fulldiameter}.
By \cite[Theorems~5.1--5.2]{AbGhDiam}, its diameter therefore differs
from the maximum by at most $O_d(1)$.

Finally, \cite[Theorem~3.2]{AbGhQuartic} describes every quartic exact
minimizer as a chain of $M_4$ middle blocks with bounded end blocks.
Each additional block adds five vertices and three to the distance
between the ends. Thus $\dm(G)=3n/5+O(1)$, which is within $O(1)$
of the maximum diameter by \cite[Theorems~5.1--5.2]{AbGhDiam}. We have used exact minimality
in this last paragraph, not just an asymptotically minimum gap.
\end{proof}

\section{Uniform bounds and relaxation time}\label{sec:evenrelaxation}
In this section, the even degree is allowed to vary with the order.
As in the uniform argument of Part~I, we treat separately the cases
$d=o(n)$ and $d$ comparable to $n$. For the first case, we extend
the shifted comparison path by constant endpoint values. For the
second, we compare the averages of a Fiedler vector on large
connected parts, retaining at least two edges across every cut
between the parts. All constants needed as the degree varies are
made explicit below.

For a $d$-regular graph, the random-walk transition matrix is $A/d$.
Its spectral gap and relaxation time are therefore $\mu(G)/d$ and
$\tau(G)=d/\mu(G)$, respectively. We first prove the uniform bound
on $n^2\mu(G)/d$, and then state the relaxation-time consequences.

\subsection{A finite-order bound}
The endpoint extension from Part~I also works with the shifted
positions in this paper. There is one additional point: the local
inequality uses the largest prefix sum in an interval. We record
that maximum by placing the whole increase of an auxiliary vector
at one position. This gives an exact comparison, without the
fixed-degree error in \Cref{lem:even-shifted-path}.

\begin{lemma}\label{lem:even-padded-bound}
Every connected simple $d$-regular graph $G$ of order $n$, with
even $d\ge6$, satisfies
\begin{equation}\label{eq:even-padded-bound}
 \mu(G)\ge 2(d-2)p_{n+6d},
 \qquad p_m=2\bigl(1-\cos(\pi/m)\bigr).
\end{equation}
The degree $d$ need not be fixed as $n$ varies.
\end{lemma}
\begin{proof}
Put $\beta=2(d-2)$, and let $\x$ be a nondecreasing mean-zero
unit Fiedler vector. Use the partition of the whole ordering in
\Cref{lem:even-certificate}. Write its intervals as
$I=\{a,\ldots,b-1\}$, put $r_I=b-a$ and
$\Delta_I=x_b-x_a$, and retain the notation
\[
 f_i=-\sum_{j=1}^i x_j,\qquad
 T_I=\frac{r_I}{\beta}+\psi_a-\psi_b.
\]
The cited lemma gives $r_I\le5d/2$, $T_I>0$, and, by
\Cref{lem:even-prefix-identities},
\begin{equation}\label{eq:even-padded-local}
 \Delta_I\le\mu T_I\max_{i\in I}f_i.
\end{equation}
Only the non-strict part of the local inequalities is used here.

At each partition endpoint $a$, prescribe the value $x_a$ at
$s_a=a-\beta\psi_a$. At the two ends $\psi_1=\psi_n=0$.
The same endpoint bounds as in the shifted-path construction give
\[
 |s_a-a|\le\frac{d-4}{2},\qquad
 \ell_I=s_b-s_a=\beta T_I>0.
\]
In particular, the positions $s_a$ remain in order, with first
position $1$ and last position $n$. Let $Y$ be the continuous
piecewise-linear interpolation through $(s_a,x_a)$, extended
constantly to the left of $1$ and to the right of $n$. Write
\[
 A=\int_1^n Y'(t)^2\,dt
   =\sum_I\frac{\Delta_I^2}{\ell_I}.
\]

For each interval choose $h_I\in I$ where $f_i$ is largest.
Define a nondecreasing vector $\bu=(u_1,\ldots,u_n)^\top$ by
$u_1=x_1$ and
\[
 u_{i+1}-u_i=
 \begin{cases}
  \Delta_I,&i=h_I\text{ for a partition interval }I,\\
  0,&\text{otherwise}.
 \end{cases}
\]
Thus $u_a=x_a$ at every original partition endpoint, and
$x_a\le u_i\le x_b$ whenever $a\le i\le b$.
Multiplying \eqref{eq:even-padded-local} by
$\Delta_I/\ell_I$ and summing gives
\begin{align}
 \beta A
 &\le\mu\sum_I\Delta_I f_{h_I}\notag\\
 &=\mu\sum_{i=1}^{n-1}f_i(u_{i+1}-u_i)
  =\mu\sum_{i=1}^n x_i u_i.
 \label{eq:even-padded-step}
\end{align}
The last equality is summation by parts, using
$f_0=f_n=0$ and $f_{i-1}-f_i=x_i$.

We now extend the comparison path. Put $L=3d$, $N=n+2L$, and
\[
 z_j=Y(j)\qquad(1-L\le j\le n+L).
\]
This is a nondecreasing vector on $N$ consecutive positions.
If $a\le i\le b$ in an original partition interval, then
$r_I\le5d/2$ and $|s_a-a|,|s_b-b|\le(d-4)/2$ imply
\[
 i-L\le s_a\le s_b\le i+L.
\]
Both $x_i$ and $u_i$ lie between $x_a=Y(s_a)$ and
$x_b=Y(s_b)$. Hence
\begin{equation}\label{eq:even-padded-brackets}
 z_{i-L}\le x_i,u_i\le z_{i+L}\qquad(1\le i\le n).
\end{equation}

Put $c=\bar z$ and $V=\sum_j(z_j-c)^2$.
For either $\bv=\x$ or $\bv=\bu$, let $k$ count the entries
$v_i\le c$. The indices $i-L$ for $i\le k$ and $i+L$ for
$i>k$ are all distinct. By \eqref{eq:even-padded-brackets},
\begin{align*}
 V&\ge\sum_{i\le k}(z_{i-L}-c)^2
       +\sum_{i>k}(z_{i+L}-c)^2\\
  &\ge\sum_{i=1}^n(v_i-c)^2
   \ge\|\bv-\bar v\1\|^2.
\end{align*}
Consequently, $V\ge\|\x\|^2=1$ and
$V\ge\|\bu-\bar u\1\|^2$. Since $\x\perp\1$,
Cauchy--Schwarz gives
\begin{equation}\label{eq:even-padded-product}
 \sum_i x_i u_i
 =\langle\x,\bu-\bar u\1\rangle\le V.
\end{equation}

On every unit interval, Cauchy--Schwarz bounds the squared difference
of its endpoint values by the integral of $Y'^2$ there. The constant
extensions add no energy. Thus the path inequality gives
\[
 p_N V\le\sum_{j=1-L}^{n+L-1}(z_{j+1}-z_j)^2\le A.
\]
Together with \eqref{eq:even-padded-step} and
\eqref{eq:even-padded-product}, this yields
$\beta p_N V\le\beta A\le\mu V$. Cancelling $V>0$ proves
\eqref{eq:even-padded-bound}.
\end{proof}

For fixed even $d\ge6$, the right-hand side of
\eqref{eq:even-padded-bound} is
$2(d-2)\pi^2/n^2+O_d(n^{-3})$. For varying degree, we instead
use the exact statement. In particular, when $d\to\infty$ and
$d=o(n)$ it gives the stronger lower limit $2\pi^2$ for
$n^2\mu(G)/d$.

\subsection{A weighted variance inequality}
We use the path version of the weighted variance inequality from
Part~I. The weights will be the proportions of vertices in the large
connected parts.

\begin{lemma}\label{lem:groupvariance}
Let $w_1,\ldots,w_k>0$, with $k\ge2$, and put
$W=\sum_iw_i$, $m=\min_iw_i$. For real numbers $z_1,\ldots,z_k$,
write $\bar z=W^{-1}\sum_iw_i z_i$. Then
\begin{equation}\label{eq:groupvariance}
 \sum_{i=1}^k w_i(z_i-\bar z)^2
 \le\frac{W^2}{8m}\sum_{h=1}^{k-1}(z_{h+1}-z_h)^2.
\end{equation}
The constant $8$ cannot be increased.
\end{lemma}
\begin{proof}
Put $b_h=z_{h+1}-z_h$. The weighted variance identity and
Cauchy--Schwarz give
\begin{align}
 \sum_iw_i(z_i-\bar z)^2
 &=\frac1W\sum_{i<j}w_iw_j(z_j-z_i)^2\notag\\
 &\le\frac1W\sum_{h=1}^{k-1}b_h^2
       \sum_{i\le h<j}w_iw_j(j-i).
 \label{eq:groupdouble}
\end{align}
Partition $[0,W]$ into consecutive intervals of lengths
$w_1,\ldots,w_k$, and let $t_i$ be their midpoints.
Since $t_{i+1}-t_i=(w_i+w_{i+1})/2\ge m$,
\[
 m(j-i)\le t_j-t_i.
\]
For a fixed $h$, put $B=\sum_{i\le h}w_i$ and $C=W-B$.
Integration of the coordinate over the two portions of $[0,W]$ gives
\[
 \sum_{i\le h}w_i t_i=\frac{B^2}{2},\qquad
 \sum_{j>h}w_j t_j=\frac{W^2-B^2}{2}.
\]
Therefore
\begin{align*}
 \sum_{i\le h<j}w_iw_j(j-i)
 &\le\frac1m\sum_{i\le h<j}w_iw_j(t_j-t_i)\\
 &=\frac1m\left(B\frac{W^2-B^2}{2}-C\frac{B^2}{2}\right)\\
 &=\frac{WBC}{2m}\le\frac{W^3}{8m}.
\end{align*}
Substitution in \eqref{eq:groupdouble} proves the inequality.
For equality take $k=2$, $w_1=w_2=1$, and $(z_1,z_2)=(-1,1)$;
both sides are $2$.
\end{proof}

\subsection{Degree comparable to the order}
We next adapt the component argument from Part~I. It is useful to
retain a general lower bound on cut sizes in this one lemma. The
application to even regular graphs will take $\kappa=2$.

\begin{lemma}\label{lem:even-dense-degree}
Fix $c>0$ and an integer $\kappa\ge1$. Let $G_n$ be connected
simple graphs of orders $n\to\infty$, with minimum degrees
$\delta_n\ge cn$, and suppose that every nontrivial edge cut has
at least $\kappa$ edges. Then
\begin{equation}\label{eq:even-dense-degree}
 \liminf_{n\to\infty}\frac{n^2\mu(G_n)}{\delta_n}\ge8\kappa.
\end{equation}
In particular, for even regular degrees $d_n\ge cn$ the lower
limit is at least $16$.
\end{lemma}
\begin{proof}
It suffices to consider a subsequence on which
$n^2\mu(G_n)/\delta_n\le C$ for a fixed $C$; a subsequence
violating the claimed lower limit would have this property.
Write $G=G_n$ and $\delta=\delta_n$. Choose a mean-zero Fiedler
vector $\y$ normalized by $\sum_v y_v^2=n$. Then
\begin{equation}\label{eq:even-dense-energy}
 E:=\sum_{uv\in E(G)}(y_u-y_v)^2=n\mu(G)
 \le C\frac\delta n\le C.
\end{equation}

Put $a_n=n^{-1/4}$, and form a spanning subgraph $H$ by retaining
exactly the edges with $|y_u-y_v|\le a_n$. At any vertex, each
deleted incident edge contributes more than $a_n^2$ to $E$.
Consequently,
\begin{equation}\label{eq:even-dense-retained}
 \delta(H)\ge\delta-C\sqrt n\ge cn/2
\end{equation}
for sufficiently large $n$. Let $V_1,\ldots,V_k$ be the connected
components of $H$, and put $n_i=|V_i|$.
Each has at least $\delta-C\sqrt n+1$ vertices, so $k\le2/c$.
The general diameter bound in \cite{Caccetta,Erdos} applies without
fixing the minimum degree. It gives
\[
 \dm(H[V_i])<\frac{3n_i}{\delta(H)+1}\le\frac6c.
\]
Every retained edge has value difference at most $a_n$. Hence,
with $\omega_n=(6/c)n^{-1/4}$,
\begin{equation}\label{eq:even-dense-oscillation}
 \max_{u,v\in V_i}|y_u-y_v|\le\omega_n.
\end{equation}

Put $w_i=n_i/n$, and let $z_i$ be the average of $\y$ on $V_i$.
The mean and norm of $\y$ imply
\begin{align}
 \sum_iw_i&=1,\qquad \sum_iw_i z_i=0,\notag\\
 \sum_iw_i z_i^2
 &=1-\frac1n\sum_i\sum_{v\in V_i}(y_v-z_i)^2
 \ge1-\omega_n^2.
 \label{eq:even-dense-variance}
\end{align}
Thus $k\ge2$ for large $n$; otherwise the sole average is zero,
contradicting \eqref{eq:even-dense-variance}. Also
\begin{equation}\label{eq:even-dense-mass}
 m:=\min_i w_i\ge\frac\delta n-\frac C{\sqrt n}.
\end{equation}

Relabel the parts so that $z_1\le\cdots\le z_k$, keeping each
weight with its part, and write $b_h=z_{h+1}-z_h\ge0$.
For each $h=1,\ldots,k-1$, choose $\kappa$ distinct original edges
crossing the cut
\[
 V_1\cup\cdots\cup V_h\,,\qquad V_{h+1}\cup\cdots\cup V_k.
\]
These edges exist by the cut hypothesis. Let $F$ be the union of
all chosen edges, counting an edge only once even when it was chosen
for several cuts. Then $|F|\le\kappa(k-1)$, and at least $\kappa$
edges of $F$ cross each of the displayed cuts.
If $uv\in F$ has endpoints in $V_i,V_j$, where $i<j$, then
\[
 (z_j-z_i)^2
 =\left(\sum_{h=i}^{j-1}b_h\right)^2
 \ge\sum_{h=i}^{j-1}b_h^2.
\]
Summing over the distinct edges of $F$ gives
\begin{equation}\label{eq:even-dense-cut-energy}
 \sum_{uv\in F}(z_{i(u)}-z_{i(v)})^2
 \ge\kappa\sum_{h=1}^{k-1}b_h^2,
\end{equation}
where $i(u)$ is the index of the part containing $u$.
An edge may cross several cuts; the preceding inequality is exactly
what permits its contribution to be used at all of them without
counting its energy more than once.

For every chosen edge, \eqref{eq:even-dense-oscillation} gives
\[
 \big|(z_{i(u)}-z_{i(v)})-(y_u-y_v)\big|\le2\omega_n.
\]
The triangle inequality for the Euclidean norm therefore gives
\begin{align*}
 \left(\sum_{uv\in F}(z_{i(u)}-z_{i(v)})^2\right)^{1/2}
 &\le\left(\sum_{uv\in F}(y_u-y_v)^2\right)^{1/2}
     +2\omega_n\sqrt{|F|}\\
 &\le\sqrt E+2\omega_n\sqrt{\kappa(k-1)}.
\end{align*}
Since $E$ and $k$ are bounded, it follows that
\begin{equation}\label{eq:even-dense-energy-comparison}
 \sum_{uv\in F}(z_{i(u)}-z_{i(v)})^2
 \le E+O_{c,C,\kappa}(\omega_n).
\end{equation}
This is a comparison of the values of one vector, not an assertion
about eigenvalues under graph contraction.

Apply \Cref{lem:groupvariance} to the weights $w_i$, whose sum is
one. From \eqref{eq:even-dense-variance}--\eqref{eq:even-dense-energy-comparison},
\begin{align*}
 E+O_{c,C,\kappa}(\omega_n)
 &\ge\kappa\sum_h b_h^2
 \ge8\kappa m\sum_iw_i z_i^2\\
 &\ge8\kappa\left(\frac\delta n-\frac C{\sqrt n}\right)
                  (1-\omega_n^2).
\end{align*}
Multiplication by $n/\delta\le1/c$ and $E=n\mu(G)$ now give
$n^2\mu(G)/\delta\ge8\kappa-o(1)$.
For even regular graphs, every nontrivial cut has positive even size,
so $\kappa=2$ gives the final assertion.
\end{proof}

\subsection{Growing even degree}
The two comparisons yield the growing-degree conclusion needed
for the uniform theorem, together with a stronger bound when
the degree is sublinear.

\begin{corollary}\label{thm:growing}
Let $G_j$ be connected simple graphs with orders $n_j\to\infty$
and even regular degrees $d_j\to\infty$. Then
\begin{equation}\label{eq:growing}
 \liminf_{j\to\infty}\frac{n_j^2\mu(G_j)}{d_j}\ge16.
\end{equation}
If also $d_j=o(n_j)$, then
\begin{equation}\label{eq:even-sublinear-growing}
 \liminf_{j\to\infty}\frac{n_j^2\mu(G_j)}{d_j}\ge2\pi^2.
\end{equation}
\end{corollary}
\begin{proof}
When $d_j=o(n_j)$, \Cref{lem:even-padded-bound} gives
\[
 \frac{n_j^2\mu(G_j)}{d_j}
 \ge2\left(1-\frac2{d_j}\right)n_j^2p_{n_j+6d_j}
 \longrightarrow2\pi^2
\]
as a lower bound. Here $d_j\ge6$ for all sufficiently large $j$.
For a general sequence, pass to a subsequence on which $d_j/n_j$
converges. If its limit is zero, use the bound just proved. If its
limit is positive, use \Cref{lem:even-dense-degree} with $\kappa=2$.
Since $2\pi^2>16$, no subsequence can have lower limit below $16$.
\end{proof}

\subsection{Sharpness at every order}
The growing-degree constant $16$ is attained by regular graphs with
two nearly equal parts and exactly two edges between them. We give
the degree and order checks explicitly.

\begin{lemma}\label{lem:even-dense-example}
For every sufficiently large $n$, there is a connected simple graph
$H_n$ of even regular degree $d_n\to\infty$ such that
\begin{equation}\label{eq:even-dense-example}
 \lim_{n\to\infty}\frac{n^2\mu(H_n)}{d_n}=16.
\end{equation}
\end{lemma}
\begin{proof}
Put $a=\lfloor n/2\rfloor$, $b=n-a$, and let $d_n$ be the largest
even integer at most $a-1$. Thus $d_n=n/2+O(1)$.
On each of the two disjoint vertex sets, of sizes $a,b$, take the
$d_n$-regular cyclic graph whose allowed differences are
$\pm1,\ldots,\pm d_n/2$. The condition $d_n\le a-1\le b-1$
makes both graphs simple and regular of the required degree.
In each graph delete the edge joining vertices $0$ and $1$.
The remaining edges of difference one still contain a spanning path,
so each part stays connected. Replace the two deleted edges by
$0_-0_+$ and $1_-1_+$ between the parts. Each affected vertex loses
one edge and gains one. The resulting graph $H_n$ is simple,
connected, and $d_n$-regular, with precisely two crossing edges.

Give the first part constant value $b$ and the second constant value
$-a$. The vector has mean zero, its energy is $2(a+b)^2=2n^2$,
and its squared norm is $ab^2+ba^2=abn$. Therefore
\[
 \mu(H_n)\le\frac{2n}{ab},\qquad
 \frac{n^2\mu(H_n)}{d_n}\le\frac{2n^3}{ab\,d_n}\longrightarrow16.
\]
Every nontrivial cut of $H_n$ has even positive size. Applying
\Cref{thm:growing} gives the matching lower limit.
\end{proof}

\begin{proof}[Proof of \Cref{thm:evenuniform}]
Fix an even $r\ge4$. Suppose that, for some $\varepsilon>0$,
there are graphs $G_j$ of orders $n_j\to\infty$ and even regular
degrees $d_j\ge r$ with
\[
 \frac{n_j^2\mu(G_j)}{d_j}<c_r-\varepsilon.
\]
Pass to a subsequence on which $d_j/n_j\to\rho\in[0,1]$.
If $\rho>0$, then $d_j\ge(\rho/2)n_j$ for large $j$.
\Cref{lem:even-dense-degree}, with $\kappa=2$, gives lower limit
at least $16\ge c_r$, a contradiction.

Suppose instead that $\rho=0$. If $d_j=4$ infinitely often,
restrict to those indices. Necessarily $r=4$, and the quartic
case of \Cref{thm:evenminimum}, which follows from the published
quartic theorem, gives lower limit at least $\pi^2=c_4$.
Otherwise we may assume $d_j\ge6$. Then
\Cref{lem:even-padded-bound} gives
\[
 \frac{n_j^2\mu(G_j)}{d_j}
 \ge2\left(1-\frac2{d_j}\right)n_j^2p_{n_j+6d_j}.
\]
Since $d_j=o(n_j)$, the second factor tends to $\pi^2$,
and $1-2/d_j\ge1-2/r$. The lower limit is therefore at least
$2(r-2)\pi^2/r\ge c_r$, again a contradiction.
This proves the lower bound uniformly over all allowed even degrees.

For $r=4,6,8,10$, the comparison graphs of degree $r$ exist at every
sufficiently large order by \Cref{lem:comparison-orders} and give
\[
 \frac{n^2\mu(G_{n,r})}{r}
 \le\frac{2(r-2)}r\pi^2+O_r(n^{-1}).
\]
For even $r\ge12$, use \Cref{lem:even-dense-example}; its degree is
at least $r$ for all sufficiently large $n$. These constructions
prove the matching upper bounds in all cases.
\end{proof}

\subsection{Consequences for relaxation time}
The cases $r=4$ and $r=6$, together with the cycle calculation, give
the following finite-threshold form of the two principal bounds.
For every $\e>0$, every sufficiently large graph of even regular
degree $d\ge2$ satisfies
\begin{align}
 \frac{n^2\mu(G)}d&\ge\pi^2-\e,\label{eq:evenuniform}\\
 \frac{n^2\mu(G)}d&\ge\frac{4\pi^2}{3}-\e
       &&\text{if }d\ge6.\label{eq:evennonquartic}
\end{align}
Indeed, a connected graph of degree two is a cycle, and
$n^2\mu(C_n)/2\to2\pi^2$. The other degrees are covered by
\Cref{thm:evenuniform}. Quartic and six-regular comparison graphs
attain the two constants, respectively.

\begin{corollary}\label{cor:evenrelaxation}
Among graphs of even regular degree and order $n$,
\[
 \max\tau=(1+o(1))\frac{n^2}{\pi^2}.
\]
Every maximizer at every sufficiently large order is quartic and
has the structure in \cite[Theorem~3.2]{AbGhQuartic}. In particular, for
odd $n$ this is also the unrestricted maximum over regular graphs.
Among graphs of even regular degree at least six,
\[
 \max\tau=(1+o(1))\frac{3n^2}{4\pi^2}.
\]
\end{corollary}
\begin{proof}
Use $\tau=d/\mu$, \Cref{thm:evenuniform}, and the matching
quartic and degree-six comparison graphs. Since $\pi^2<4\pi^2/3$
and cycles have coefficient $2\pi^2$, every sufficiently large
even-degree maximizer is quartic. Its structure is the cited
published theorem, without any additional uniqueness claim about
the end blocks. Odd order forces even regular degree by the
handshaking identity.
\end{proof}


\begin{thebibliography}{99}


\bibitem{AbGhCompanion}
M. Abdi and E. Ghorbani,
Graphs with Minimum Algebraic Connectivity I: Proofs of
Aldous--Fill and Guiduli--Mohar Conjectures,
Preprint, September 2026.

\bibitem{AbGhMin}
M. Abdi and E. Ghorbani,
Graphs of degree at least $3$ with minimum algebraic connectivity,
{\em SIAM J. Discrete Math.} {\bf 38} (2024), 2447--2467.

\bibitem{AbGhDiam}
M. Abdi and E. Ghorbani,
Minimum algebraic connectivity and maximum diameter:
Aldous--Fill and Guiduli--Mohar conjectures,
{\em J. Combin. Theory Ser. B} {\bf 167} (2024), 164--188.

\bibitem{AbGhQuartic}
M. Abdi and E. Ghorbani,
Quartic graphs with minimum spectral gap,
{\em J. Graph Theory} {\bf 102} (2023), 205--233.

\bibitem{AbGhIm}
M. Abdi, E. Ghorbani and W. Imrich,
Regular graphs with minimum spectral gap,
{\em European J. Combin.} {\bf 95} (2021), Article 103328.

\bibitem{AldousFill}
D. Aldous and J. Fill,
{\em Reversible Markov Chains and Random Walks on Graphs},
Unfinished monograph, 2002; recompiled 2014.


\bibitem{Imrich}
C. Brand, B. Guiduli and W. Imrich,
Characterization of trivalent graphs with minimal eigenvalue gap,
{\em Croat. Chem. Acta} {\bf 80} (2007), 193--201.

\bibitem{Caccetta}
L. Caccetta and W. F. Smyth,
Graphs of maximum diameter,
{\em Discrete Math.} {\bf 102} (1992), 121--141.

\bibitem{DoyleSnell}
P. G. Doyle and J. L. Snell,
{\em Random Walks and Electric Networks},
Mathematical Association of America, 1984;
online version dated July 5, 2006,
available at \url{https://math.dartmouth.edu/~doyle/docs/walks/walks.pdf}.

\bibitem{Erdos}
P. Erd\H{o}s, J. Pach, R. Pollack and Z. Tuza,
Radius, diameter, and minimum degree,
{\em J. Combin. Theory Ser. B} {\bf 47} (1989), 73--79.

\bibitem{fiedler1973algebraic}
M. Fiedler,
Algebraic connectivity of graphs,
{\em Czechoslovak Math. J.} {\bf 23} (1973), 298--305.

\bibitem{GuiduliThesis}
B. Guiduli,
{\em Spectral Extrema for Graphs},
Ph.D. Thesis, University of Chicago, 1996.
\bibitem{Guiduli}
B. Guiduli,  The structure of trivalent graphs with minimal eigenvalue gap, {\em J. Algebraic Combin.} {\bf6} (1997), 321--329.
 
\bibitem{Spielman}
D. A. Spielman,
{\em Spectral and Algebraic Graph Theory}, Yale University,
incomplete draft dated April 2, 2025,
available at \url{https://www.cs.yale.edu/homes/spielman/sagt/sagt.pdf}.

\bibitem{Zhu}
H. Zhu,
The maximum relaxation time of a random walk on regular graphs,
arXiv:2609.06818v1, 2026.

\end{thebibliography}
\end{document}